\documentclass[oneside,reqno]{amsart}
\usepackage{amssymb,amsmath}
\usepackage[english]{babel}
\usepackage{amsfonts}
\usepackage{color}
\usepackage{amsthm}
\usepackage[colorlinks=true,linkcolor=NavyBlue, citecolor=NavyBlue,urlcolor=NavyBlue]{hyperref}
\usepackage{graphicx}
\usepackage{mathtools}
\usepackage[usenames,dvipsnames]{pstricks}
\usepackage{bm}
\usepackage{amsmath}
\usepackage{listings}
\usepackage{hyperref}
\usepackage{tikz}
\usepackage{tikz-cd}
\usepackage{cancel}

\usepackage{xy}
\usepackage[draft]{fixme}
\DeclareMathOperator{\AGL}{AGL}
\DeclareMathOperator{\GL}{GL}
\DeclareMathOperator{\Sym}{Sym}

\newcommand*{\pmat}[1]{\begin{bmatrix}#1\end{bmatrix}}

\newcommand*{\mat}[1]{\begin{matrix}#1\end{matrix}}
\newcommand*{\spmat}[1]{\left(\begin{smallmatrix}#1\end{smallmatrix}\right)}

\DeclareMathOperator{\Ann}{Ann}
\DeclareMathOperator{\gl}{GL}
\DeclareMathOperator{\agl}{AGL}
\DeclareMathOperator{\aut}{Aut}
\DeclareMathOperator{\rank}{Rank}
\DeclareMathOperator{\sym}{Sym}
\DeclareMathOperator{\syp}{Sp}
\DeclareMathOperator{\fix}{Fix}
\DeclareMathOperator{\soc}{Soc}

\newcommand{\p}{\circ}

\let\phi\varphi
\theoremstyle{plain}
\newtheorem{theorem}{Theorem}[section]

\newtheorem{proposition}[theorem]{Proposition}
\newtheorem{corollary}[theorem]{Corollary}

\theoremstyle{remark}
\newtheorem{remark}{Remark}

\theoremstyle{definition}
\newtheorem{definition}[theorem]{Definition}
\newtheorem{example}[theorem]{Example}
\newtheorem*{notation*}{Notation}

\usepackage{listings}
\lstdefinelanguage{GAP}{%
 morekeywords={%
 Assert,Info,IsBound,QUIT,%
 TryNextMethod,Unbind,and,break,%
 continue,do,elif,%
 else,end,false,fi,for,%
 function,if,in,local,%
 mod,not,od,or,%
 quit,rec,repeat,return,%
 then,true,until,while%
 },%
 sensitive,%
 morecomment=[l]\#,%
 morestring=[b]",%
 morestring=[b]',%
}[keywords,comments,strings]

\usepackage[T1]{fontenc}
\usepackage[variablett]{lmodern}
\usepackage{xcolor}
\usepackage{enumerate}
\DeclareMathAlphabet{\mbb}{U}{bbold}{m}{n}
\newcommand{\ear}{\texttt{EARS}}
\newcommand{\baa}{\texttt{BAA}}
\newcommand{\bbb}{\texttt{BBB}}
\newcommand{\D}{\Delta}
\usepackage{mathtools}
\usepackage{orcidlink}

\begin{document}
\title{Binary bi-braces and applications to cryptography} 
\author[R.~Civino]{Roberto Civino 
\orcidlink{0000-0003-3672-8485}
}
\author[V.~Fedele]{Valerio Fedele}

\address{DISIM \\
 Universit\`a degli Studi dell'Aquila\\
 via Vetoio\\
 67100 Coppito (AQ)\\
 Italy}       

\email{roberto.civino@univaq.it}
\email{valerio.fedele@guaduate.univaq.it}

\date{} \thanks{R.\ Civino is member of INdAM-GNSAGA
 (Italy) and is supported by the Centre of EXcellence on Connected, Geo-Localized and 
 Cybersecure Vehicles (EX-Emerge), funded by Italian Government under CIPE resolution n.\ 70/2017 (Aug.\ 7, 2017).}


\begin{abstract}
In a XOR-based alternating block cipher the plaintext is masked by a sequence of 
layers each performing distinct actions: a highly nonlinear permutation, a linear transformation, and the bitwise key addition.
When assessing resistance against \emph{classical} differential attacks (where differences are computed with respect to XOR), the cryptanalysts must only take into account 
differential probabilities introduced by the nonlinear layer, this being the only one whose differential transitions are not
deterministic. The temptation of computing differentials with respect to another difference operation 
runs into the difficulty of understanding how differentials propagate through the XOR-affine levels of the cipher.
In this paper we introduce a special family of \emph{braces} that enable the derivation of a set of 
differences whose interaction with \emph{every} layer of an XOR-based alternating block cipher can be understood.
We show that such braces can be described also in terms of alternating binary algebras of nilpotency class two. 
Additionally, we present a method to compute the automorphism group of these structures through an 
equivalence between bilinear maps. By doing so, we characterise the XOR-linear permutations for which 
the differential transitions with respect to the new difference are deterministic, facilitating an 
alternative differential attack.
\end{abstract}

\maketitle

\section{Introduction}
Let $N$ be an integer such that computing $2^N$ operations is unfeasible and let $M$ be an $N$-dimensional vector space over $\mathbb F_2$. Generally speaking, a block cipher $\Phi \subset \Sym(M)$ is a large subset of permutations $E_k$ over the \emph{message space} $M$, where $k$ ranges in a set $\mathcal K$ of parameters called the \emph{key space}. 
For the purposes of this paper, it is sufficient to assume that $\Phi$ is a key-alternating XOR-based cipher from the family of \emph{substitution-permutation networks} (SPN), namely a set of permutations in which the encryption function is obtained as the composition of an iterated structure made by a \emph{confusion layer} which is a highly nonlinear function, by a linear \emph{diffusion layer} which maximises the effects of a local perturbation in the plaintext and by the XOR-based key addition with the so-called \emph{round key}. The public components of such ciphers are designed on one hand keeping the focus on the efficiency of the whole encryption/decryption process, and on the other on preventing all possible known attack. The most famous is undoubtably \emph{differential cryptanalysis}~\cite{biham1991differential,K94,W99,BCJW02,BBS99}, a general chosen-plaintext technique against symmetric primitives in which the advantage coming from the ability to predict the output difference of encrypted messages  allows the cryptanalyst to recover information on the supposedly unknown key. After decades of research, block-cipher designers are equipped with methods which make the attack unfeasible, i.e.\ the high nonlinearity of the confusion layer and the diffusion property of the linear layer make biases in the distribution of difference of ciphertexts very hard to predict.
In both cases, the nonlinearity of the confusion layer and the capacity of the diffusion layer to spread the differences are measured with respect to the key-mixing operation, i.e.\ the XOR. As a consequence, the family of layers that nowadays we find in standardised constructions are somewhat optimal in this sense. As an example, the confusion layer of the block cipher PRESENT~\cite{bogdanov2007present} features a 4-bit permutation whose nonlinearity w.r.t.\ the XOR is optimal in the sense of Nyberg~\cite{nyberg1991perfect,N93} and Leander and Poschmann~\cite{leander2007classification}, 
and the confusion layer of AES~\cite{daemen2002design} uses a widely-believed (XOR)-optimal 8-bit permutation.\\

The idea of generalising the classical differential attack to new operations, in particular when applied to a cipher utilising XOR-based key addition, is not new. 
For instance, Berson introduces the concept of modular difference to analyse the MD/SHA family of hash functions, and a similar approach has been employed to cryptanalyse the block cipher PRESENT \cite{B92}. Borisov et al.~proposed a novel form of differential known as multiplicative differential to target IDEA~\cite{IDEA}. This led to the formulation of $c$-differential uniformity \cite{CDIFF}, which has been extensively explored in recent years, though its cryptographic implications on attacking block ciphers remain a matter of ongoing discussion due to the challenge in understanding how difference propagation happens through all the layers of the cipher~\cite{DANIELE}.
Te{\c{s}}eleanu considered the case of SPNs and operations coming from quasigroups~\cite{tecseleanu2022security}.

Calderini et al.~\cite{calderini2021properties}, with the idea of designing a technique in which the attacker embeds the cipher in an isomorphic copy of the affine group $\AGL(M)$, provided an explicit construction of a family of operations coming from elementary abelian regular subgroups of $\AGL(M)$ or, in short, \emph{translation groups}.
A series of papers have shown that the operations coming from translation groups can be employed in an attack, called \emph{alternative differential cryptanalysis}, against XOR-based ciphers resistant to classical differential attacks~\cite{CBS19,calderini2024differential,calderini2024optimal}. 
In precise terms, it is demonstrated that certain operations, represented by $\circ$ and induced by translation groups, allow for a higher predictability of the difference between many pairs of messages of the form $(x, x\circ \Delta)$ with a fixed difference $\Delta$. This predictability is higher when computing the difference after encryption with respect to $\circ$ compared to replacing $\circ$ with XOR (denoted later simply by +) for the same differences.\\

This paper extensively explores the algebraic aspects associated with these operations. We recall the properties that the operations coming from translation group must satisfy in order to make the \emph{alternative} differential attack feasible, and turn them into axioms for an algebraic structure, that we call \emph{binary bi-brace}. The resultant algebraic structure is a particular instance of \textit{skew (left) braces} \cite{GV16}, which generalise \textit{(left) braces} initially introduced by Rump \cite{Ru07} as an extension of radical rings to investigate nondegenerate involutive set-theoretic solutions of the Yang–Baxter equation~\cite{MR1722951}. A comprehensive examination of skew braces' properties followed due to their prevalence across various mathematical domains. Skew braces can be described in terms of regular subgroups of the holomorph, as discussed in many works~\cite{Ca05,CR09,GV16,CARANTI2020647}. Furthermore, the exploration of skew braces is intimately connected to the study of Hopf–Galois structures on finite Galois extensions~\cite{SV18,ST23a}.

We make explicit the connection between the family of binary bi-braces $(R,+,\circ)$, the family of translation groups characterised by a mutual normalisation property (as in Calderini et al.), and a family of alternating algebras of nilpotency class two $(R,+,\cdot)$, insisting on the same binary vector space structure $(R,+)$. 
Moreover, building on an observation in Civino et al.~\cite{CBS19}, we frame the investigation within the context where differences propagate with the highest probability through the key-addition layer of the cipher, i.e.\ when $\dim(R\cdot R) = 1$. In this setting, we provide a constructive characterisation of the diffusion layers for which the attack becomes feasible, which in algebraic terms corresponds to understanding the automorphism group of the algebra. 

\subsection*{Organisation of the paper}
The remainder of the paper is organised as follows. 
In Sect.~\ref{sec:block}, following the definition of a family of block ciphers, we briefly review the assumptions underlying the classical and alternative forms of the differential attack, where the latter employs operations from other translation groups.
In Sect.~\ref{sec:braces} we introduce binary bi-braces and prove their equivalence with other algebraic structures. We elucidate the significance of the knowledge of the automorphism groups of these structures in cryptography. Particularly, binary bi-braces are proven to be equivalent to alternating algebras of nilpotency class two, discussed in a broader context in Sect.~\ref{sec:algebras} utilising bilinear maps.
Building upon the characterisation of isomorphisms between alternating algebras established in Sect.~\ref{sec:algebras} (refer to Theorem~\ref{congruent}) and their automorphisms (see~Corollary~\ref{cor:aut}), in Sect.~\ref{sec:binalg} we describe the automorphism group of a binary bi-brace specifically when $\dim(R^2)=1$ (see~Corollary~\ref{SpFix}), which is the pivotal case for applications. We also prove how to extend the knowledge of the automorphism group of a binary bi-brace to the automorphism group of the direct product of such structures (see~Theorem~\ref{thm:parallelgen} and Theorem~\ref{thm:parallelone}).
We conclude the paper with Sect.~\ref{sec:concl}, containing some cryptographic considerations regarding the outcomes of our research.


\section{Motivation and preliminaries on block ciphers}\label{sec:block}
In this section, we  formalise the concept of the specific type of block ciphers that are of interest to us, and we introduce the fundamental concepts of differential cryptanalysis required for our discussion. We  elucidate the requirements for the new operation coming from a translation group to enable us to replicate the classical attack.

\subsection{SPNs and classical differential cryptanalysis}

 Let $\dim(M)=N=n\times h$ and let us write $M= R\oplus R \oplus \ldots \oplus R$ where $\dim(R) = n$ for $1\leq j\leq h$, and $\oplus$ represents the direct sum of $n$-dimensional subspaces, called \emph{bricks}. Let \[T_+(R) =\{\sigma_y : y\in R\} < \Sym(R)\]where $\sigma_y$ is such that $x\mapsto x+y$ and let $T_+(M)$ be defined accordingly. Notice that $T_+(M)$ corresponds to the direct product  $T_+(R) \times  T_+(R) \times \cdots \times T_+(R)$.

\begin{definition}\label{def:cipher}
Let $r \in \mathbb N$.
An \emph{$r$-round substitution-permutation network} over $M$ and over a key space $\mathcal K$ is a family of (encryption) functions \[\{E_k :  k \in \mathcal K\, \} \subset \Sym(M)\] such that  for each $k \in \mathcal K$ the map $E_K$ is the composition of $r$ \emph{round functions}, i.e.\ 
\[E_k = E_{k,1}\:E_{k,2}\ldots E_{k,r},\]
where $E_{k,i} = \gamma \mu \sigma_{k_i}$ and 
\begin{itemize}
\item $\gamma \in \Sym(M)$ is a nonlinear transformation which operates in a \emph{parallel way} with respect to the decomposition on $M$, acting of every brick as a permutation $\gamma'$, i.e.\ \[(x_1,x_2,\ldots,x_N)\gamma = \left((x_1,\ldots,x_{n})\gamma',\ldots,(x_{n(h-1)+1},\ldots,x_{N})\gamma'\right).\] The map $\gamma' \in \Sym(R) $ is traditionally called an \emph{s-box} and $\gamma$ the \emph{confusion layer};
\item $\mu \in \GL(M)$ is traditionally called the \emph{diffusion layer};
\item $\sigma_{k_i} \in T_+(M)$ represents the key addition, where $+$ is the usual bitwise XOR on $\mathbb F_2$. The round keys $k_i \in V$ are usually derived from the master key $k \in \mathcal K$ by means of a public algorithm, called \emph{key schedule}. 
We omit the details about key schedules here since they are not relevant to the current discussion.
\end{itemize}
The general round function for an SPN is depicted in Fig.~\ref{fig.cipher}

\begin{figure}
\centering
\scalebox{1.3}{
\begin{picture}(140,70)
\put(0,50){\line(0,1){20}}
\put(20,50){\line(0,1){20}}
\put(0,50){\line(1,0){20}}
\put(0,70){\line(1,0){20}}
\put(30,50){\line(0,1){20}}
\put(50,50){\line(0,1){20}}
\put(30,50){\line(1,0){20}}
\put(30,70){\line(1,0){20}}
\put(60,50){\line(0,1){20}}
\put(80,50){\line(0,1){20}}
\put(60,50){\line(1,0){20}}
\put(60,70){\line(1,0){20}}
\multiput(85,60)(5,0){6}{\put(0,0){\line(1,0){2}}}

\put(120,50){\line(0,1){20}}
\put(140,50){\line(0,1){20}}
\put(120,50){\line(1,0){20}}
\put(120,70){\line(1,0){20}}

\multiput(2,70)(30,0){3}{\multiput(0,0)(4,0){2}{\put(0,0){\line(0,1){4}}}}
\multiput(8,72)(30,0){3}{\multiput(0,0)(3,0){2}{\put(0,0){\line(1,0){1}}}}
\multiput(14,70)(30,0){3}{\multiput(0,0)(4,0){2}{\put(0,0){\line(0,1){4}}}}

\multiput(122,70)(4,0){2}{\put(0,0){\line(0,1){4}}}
\multiput(128,72)(3,0){2}{\put(0,0){\line(1,0){1}}}
\multiput(134,70)(4,0){2}{\put(0,0){\line(0,1){4}}}

\multiput(2,40)(30,0){3}{\multiput(0,0)(4,0){2}{\put(0,0){\line(0,1){10}}}}
\multiput(8,45)(30,0){3}{\multiput(0,0)(3,0){2}{\put(0,0){\line(1,0){1}}}}
\multiput(14,40)(30,0){3}{\multiput(0,0)(4,0){2}{\put(0,0){\line(0,1){10}}}}

\multiput(122,40)(4,0){2}{\put(0,0){\line(0,1){10}}}
\multiput(128,45)(3,0){2}{\put(0,0){\line(1,0){1}}}
\multiput(134,40)(4,0){2}{\put(0,0){\line(0,1){10}}}

\put(0,20){\line(0,1){20}}
\put(140,20){\line(0,1){20}}
\put(0,20){\line(1,0){140}}
\put(0,40){\line(1,0){140}}

\multiput(2,0)(30,0){3}{\multiput(0,0)(4,0){2}{\put(0,0){\line(0,1){20}}}}
\multiput(8,15)(30,0){3}{\multiput(0,0)(3,0){2}{\put(0,0){\line(1,0){1}}}}
\multiput(14,0)(30,0){3}{\multiput(0,0)(4,0){2}{\put(0,0){\line(0,1){20}}}}
\multiput(122,0)(4,0){2}{\put(0,0){\line(0,1){20}}}
\multiput(128,15)(3,0){2}{\put(0,0){\line(1,0){1}}}
\multiput(134,0)(4,0){2}{\put(0,0){\line(0,1){20}}}

\multiput(2,10)(30,0){3}{\multiput(0,0)(4,0){2}{\put(0,0){\circle{3}}}}
\multiput(0.5,10)(30,0){3}{\multiput(0,0)(4,0){2}{\put(0,0){\line(1,0){3}}}}
\multiput(14,10)(30,0){3}{\multiput(0,0)(4,0){2}{\put(0,0){\circle{3}}}}
\multiput(12.5,10)(30,0){3}{\multiput(0,0)(4,0){2}{\put(0,0){\line(1,0){3}}}}

\multiput(122,10)(4,0){2}{\put(0,0){\circle{3}}}
\multiput(120.5,10)(4,0){2}{\put(0,0){\line(1,0){3}}}
\multiput(134,10)(4,0){2}{\put(0,0){\circle{3}}}
\multiput(132.5,10)(4,0){2}{\put(0,0){\line(1,0){3}}}

\put(-10,60){\makebox(0,0){$\gamma$}}
\put(-10,30){\makebox(0,0){$\mu$}}
\put(-10,10){\makebox(0,0){$\sigma_{k_i}$}}

\multiput(10,60)(30,0){3}{\put(0,0){\makebox(0,0){$\gamma'$}}}
\put(130,60){\makebox(0,0){$\gamma'$}}
\end{picture}}

\caption{\label{fig.cipher}Example of 1-round encryption}
\end{figure}
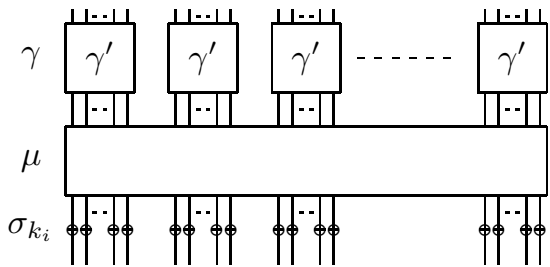
\end{definition}

The concept behind classical differential cryptanalysis involves the ability to identify a statistical bias in the distribution of the sets of differences $\{xE_k + (x+\Delta)E_k : x \in M\}$, given a certain $\Delta \in M$, across a substantial number of keys.
Specifically, referring to a pair of differences $(\Delta_I, \Delta_O)$ as a \emph{differential}, we define the \emph{differential probability} of the pair $(\Delta_I, \Delta_O)$ as
\[ \mathbb P_{x,k} \left(x E_k+(x+\Delta_I)E_k = \Delta_O\right) = \mathbb P_{x,k} \left(\Delta_I\xrightarrow[]{E_k,+}\Delta_O\right),\]
under the assumption that $x$ and $k$ are uniformly distributed. This probability represents the likelihood that, given two vectors whose difference is $\Delta_I$, the resulting difference after applying  a random $E_k$ is $\Delta_O$.
Since computing the distribution of $\{xE_k+(x+\Delta)E_k: x \in M\}$ is computationally unfeasible even for a single key, we follow the customary approach found in the literature~\cite{lai1991markov}, which involves assuming independence among the round keys and computing the probability that $\Delta_I \xrightarrow[]{E_K}\Delta_O$ as the product of the probabilities \[\prod_{i=1}^r \mathbb P_{x,k}\left( \Delta_I \xrightarrow[]{E_{k,i}} \Delta_O\right),\] where each factor is determined by the product of the probabilities of the difference propagating through the confusion, diffusion, and key-addition layers.
It is worth noticing here that, in the classical attack, only the nonlinear confusion layer gives a nontrivial contribution in terms of differential probabilities. Indeed, if $(\Delta_I, \Delta_O)$ is a differential, then 
\begin{equation}\label{eq:diffdiff}
\mathbb P_x \left(\Delta_I \xrightarrow[]{\mu} \Delta_O\right) =  \mathbb P_x \left(x \mu + (x+\Delta_I) \mu = \Delta_O \right)
\end{equation}
is 1 if $\Delta_I\mu=\Delta_O$ and 0 otherwise. Similarly, differentials propagate deterministically though the key-addition layer, being 
\begin{equation}\label{eq:keydiff}
x\sigma_{k_i} + (x+\Delta)\sigma_{k_i} = x+k_i+x+\Delta+k_i = \Delta,
\end{equation}
for each $x, k_i \in M$.
\subsection{Alternative operations from translation groups}\label{sec:crittoreq}
In this paper, we aim to classify new operations $\circ$ on $M$ in a way that enables the detection of biases in the set of differences $\{xE_k \circ (x\circ \Delta)E_k : x \in M\}$. Such biases are obtained in terms of  the \emph{differential probabilities with respect to $\circ$} that can be defined similarly to before, but with the replacement of $+$ with $\circ$.

To achieve this, we operate under the assumptions specified below~\cite{CBS19,calderini2024differential}, which are later formalised as the axioms of binary bi-braces.
\subsubsection*{Parallel operations from translation groups}
We initially assume that $\circ$ is induced by another translation group $T_\circ(M)$, denoted as $T_\circ(M) = \{\tau_a: a \in M\}$, which is elementary abelian and regular, and where $a \circ b = a\tau_b$, with $\tau_b \in T_\circ(M)$ being the unique element that maps $0$ to $b$.
Moreover, given that the cipher incorporates a confusion layer operating in a parallel way on the bricks and that bitwise key addition can similarly be viewed as a parallel process on the bricks, the operation $\circ$ is crafted at the brick level. This means it is induced by a translation group $T_\circ(R) < \Sym(R)$ and then extended to $M$ in the straightforward manner.
In this context, for the sake of simplicity in notation, we denote the (parallel) extension of $\circ$ to $M$ by the same symbol. Therefore, if $a, b \in M$, then $a \circ b$ represents the operation obtained by applying the translations of $T_\circ(R)$ to every brick of $a$ and $b$,
and we say that $\circ$ is a \emph{parallel} operation on $M$.
\subsubsection*{Affine groups} Following the methodology as in Calderini et al.~\cite{calderini2021properties}, to facilitate a constructive representation of translations in $T_\circ(R)$ using matrices, we adopt the assumption that $T_\circ(R)$ forms an affine group. Specifically, we require that $T_\circ(R)< \AGL(R,+)$. This also implies that  $T_\circ(M)< \AGL(M,+)$. 
\subsubsection*{Dual normalisation}
The aforementioned requirement can be equivalently expressed by demanding that $T_\circ(M)$ normalises $T_+(M)$. This condition alone does not address one significant limitation of the attack approach, i.e.\ the challenge of the prediction of the differential probabilities w.r.t.\ $\circ$ through the key-addition layer:
\begin{equation}\label{eq:keydiffcirc}
\mathbb P_{x,k_i} \big(x\sigma_{k_i} \circ (x\circ\Delta)\sigma_{k_i} = (x+k_i)\circ \left((x\circ\Delta)+k_i\right) = \Delta_O\big).
\end{equation}
Indeed, it is worth noting that, unlike in the classical case (see~Eq.\eqref{eq:keydiff}), the previous probability depends on both $x$ and the round key $k_i$.
The problem of unpredictability concerning the equation above can solved, as demonstrated by Civino et al.~\cite[Theorem 3.12]{CBS19}, through the introduction of a second normalisation condition.
Specifically, they require  $T_\circ(R)$ to be normalised by $T_+(R)$, i.e.\ $T_+(R) \leq N_{\Sym(R)}\left(T_\circ(R)\right)$, and therefore that $T_+(M) \leq N_{\Sym(M)}\left(T_\circ(M)\right)$. Defining $R^2$ as
$R^2 = <x \cdot y : x, y \in R >$, where 
\begin{equation}\label{eq:infoprod}
x \cdot y = x+y+x\circ y,
\end{equation} 
the dual normalisation property implies that
\begin{equation}\label{eq:Uone}
(x+k_i)\circ((x\circ\Delta)+k_i)= \Delta + k_i \cdot \Delta \in \Delta + R^2.
\end{equation}
The equation above indicates two important facts. Firstly, it demonstrates that the output difference through the key-addition layer, calculated with respect to the operation $\circ$, does not depend on the potentially unknown value of $x$. Secondly, it suggests that under the assumption that the subspace $R^2$ is small, the actual value of Eq.\eqref{eq:Uone} can be reliably predicted with high probability. Thus, the most pertinent scenario, to be explored in this paper due to its significance in cryptanalysis, occurs when $\dim(R^2)=1$, implying that the value of Eq.\eqref{eq:Uone} can be predicted with a probability of $1/2$. Importantly, for half of the inputs, this is equal to $\Delta$, just as in the classical case.\\

We have briefly discussed how the aforementioned assumptions are relevant in the cryptographic context. Additionally, it is important to note that the success of the attack, namely the ability to predict differentials with respect to $\circ$ with high probability, crucially depends on the predictability of their propagation through the diffusion layer of the cipher. Recall that in the classical context, XOR-differentials propagate through the diffusion layer with a probability of one (see\ Eq.~\eqref{eq:diffdiff}). However, now, the presence of the function $\mu \in \GL(M)$ poses a significant challenge. The diffusion layer, indeed, is not linear with respect to the operation $\circ$, necessitating the introduction of new differential probabilities.
Such probabilities have to do with the possibility to predict if $\Delta_I \xrightarrow[]{\mu,\circ} \Delta_O$, i.e.\ to compute 
\[
\mathbb P_x \big(x\mu \circ (x\circ\Delta_I)\mu = \Delta_O\big),
\]
which is expected to be very low due to the inability to isolate the action of $\mu$ at the individual brick level; instead, each $x \in M$ must be considered at the block level.
 A potential solution to this issue arises from the assumption that the diffusion layer of the cipher behaves linearly with respect to both the operations $+$ and $\circ$. This ensures the validity of the result in Eq.~\eqref{eq:diffdiff} even when $+$ is replaced by $\circ$. For this reason, determining which maps satisfy the condition of linearity with respect to both operations is essential. 
From the perspective of a cryptanalyst, these maps could serve as trapdoor diffusion layers within the cipher, particularly in relation to a potential differential attack exploiting the operation $\circ$. Consequently, understanding their characteristics holds significant importance.
 In Sect.~\ref{sec:binalg}, for the benefit of cryptanalysts, we provide a characterisation of linear maps with respect to $+$ and $\circ$, where $\circ$ is an operation defined on $R$ satisfying $\dim(R^2)=1$, which is then extended to $M$ in a parallel way.


\section{On braces and binary structures}\label{sec:braces}
As outlined in Section~\ref{sec:crittoreq}, the operation employed for the alternative differential attacks stems from pairs of translation groups (one of which is fixed to be $T_+(R)$) featuring a mutual normalisation property. We introduce the family of such groups, along with a set of diverse algebraic constructions that we show to be equivalent.
One equivalent representation of these operations is through binary bi-braces, that are defined shortly. Let us first revisit the definitions of braces, starting from a nonempty set $B$.

\begin{definition}[\cite{GV16}]
 A \emph{skew (left) brace} is a triple $(B,+,\circ)$ such that $(B,+)$ and $(B,\circ)$ are (not necessarily abelian) groups and for each $a,b,c \in B$
\begin{align}\label{sbdistr}
	a\circ(b+c)=(a\circ b)-a+(a\circ c).
\end{align}
The groups $(B,+)$ and $(B,\circ)$ are respectively called the \emph{additive} and \emph{multiplicative} group of the skew brace $B$.
\end{definition}
As a consequence of Eq.~\eqref{sbdistr}, the neutral elements of
$(B,+)$ and $(B,\circ)$ coincide and are denoted here by $0$. By $a^{-1}$ we denote the inverse of $a \in B$ with respect to the operation $\circ$.
The following definition was given by Childs~\cite{Ch19} and subsequently studied by different authors~\cite{CARANTI2020647,CS21,ST23}.
\begin{definition}[\cite{Ch19}]
A skew brace $(B,+,\circ)$ is called a \emph{bi-skew brace} if also $(B,\circ,+)$ is  a skew brace, i.e.\ for each $a,b,c \in B$
\[a+(b\circ c)=(a+b)\circ a^{-1}\circ(a+c).\]
A skew brace with an abelian additive group is called a \emph{(left) brace}. 
Similarly, a \emph{(right) brace} can be defined replacing the previous condition by
\[(a+b)\circ c=(a\circ c)-c+(b\circ c).\]
A \emph{(two-sided) brace} is a left brace $(B,+,\circ)$ that is also a right brace. 
\end{definition}
Clearly, a brace with an abelian multiplicative group $(B,\circ)$ is a two-sided brace.

\begin{definition}[\cite{CCS19}]\label{def:fbra}
Let $F$ be a field. A \emph{(left) $F$-brace} $(B,+,\circ)$ is a brace whose additive group possesses a vector space structure over $F$ and  for $a,b,c\in B,\,\delta\in F$
\begin{align}\label{fbrace}
	\delta(a\circ b)=\left(a\circ (\delta b)\right)+(\delta-1)a.
\end{align}
\end{definition}
It is also worth mentioning that the concept of a brace over a field has recently been generalised through the introduction of a \emph{module brace} over a ring~\cite{DC23}.

Let us revisit some established results regarding the equivalence between braces and other algebraic structures.

\subsection{Some known equivalences}

Let us start by recalling the concept that a ring $R$ (respectively an algebra $R$ over a field $F$) is radical if every element is invertible with respect to the circle operation $x\circ y=x+y+x\cdot y$, i.e.\ if $(R,\circ)$ is a group. The best known examples of radical rings (resp.\ radical algebras) are the nilpotent rings (resp.\ nilpotent algebras), where a ring $R$ (resp.\ an $F$-algebra $R$) is nilpotent if there exists a positive
integer $c$ such that every product of $c+1$ elements from $R$ is $0$, and the smallest $c$ with this property is called the nilpotency class of $R$. In particular, if $c$ is the nilpotency class of $R$, then $R^{c+1} = 0$ and $R^c\neq 0$.

Caranti et al.~\cite{Ca05} obtained a simple description of the abelian regular subgroups of the affine group $\AGL(R)$ in terms of commutative associative radical $F$-algebras $(R,+,\cdot)$ with underlying vector space $R$. 

\begin{theorem}{\cite[Theorem 1]{Ca05}}\label{tca05}
	Let $R$ be a vector space over a field $F$. Denote by \emph{\texttt{RA}} the family of commutative associative radical $F$-algebras with underlying vector space $R$ and by \emph{\texttt{T}} the family of all abelian regular subgroups of the affine group $\agl(R)$.
	\begin{enumerate}
		\item  If $(R,+,\cdot)$ is a commutative associative radical $F$-algebra with underlying vector space $R$, then $T_\circ(R)=\{\tau_y:\,y\in R\}$, where $\tau_y:R\longrightarrow R,\ x\mapsto x\circ y$ is an abelian regular subgroup of the affine group $\agl(R)$.
		\item  The map
		\[f: \text{\emph{\texttt{RA}}}\longrightarrow\text{\emph{\texttt{T}}},\ (R,+,\cdot)\longmapsto T_\circ(R)\]
		is a bijection. Under this correspondence, isomorphism classes of commutative associative $F$-algebras correspond to conjugacy classes under the action of $\gl(R)$ of abelian regular subgroups of $\agl(R)$.
	\end{enumerate}
\end{theorem}

Additionally, a bijective correspondence exists between radical rings and two-sided braces, as demonstrated by Rump (see~also Ced\'{o}~\cite{Ce18}).

\begin{theorem}{\cite[Proposition 2.4]{Ce18}}\label{Rump}
	If $(B,+,\circ)$ is a two-sided brace, then $(B,+,\cdot)$ is a radical ring, where
	$x\cdot y=x\circ y -x-y$ for $x,y\in B$.
	Conversely, if $(B,+,\cdot)$ is a radical ring, then $(B,+,\circ)$ is a two-sided brace, where $x\circ y=x+y + x\cdot y$ for $x,y\in B$.
\end{theorem}

Notice that for a two-sided $F$-brace Eq.~\eqref{fbrace} is equivalent to the compatibility with scalars of the corresponding radical $F$-algebra, i.e. $\delta(x\cdot y)=x\cdot(\delta y)=(\delta x)\cdot y$. 

The result of Theorem~\ref{tca05} has been extended by Catino and Rizzo~\cite{CR09}, who established a link between (not necessary abelian) regular subgroups of the affine group $\agl(R)$ and $F$-braces structures with the underlying vector space $R$.

\begin{theorem}{\cite{CR09}}\label{thm:CR}
	Let $R$ be a vector space over a field $F$. Denote by \emph{\texttt{RB}} the family of $F$-braces with underlying vector space $R$ and by \emph{\texttt{T}} the family of all regular subgroups of the affine group $\agl(R)$.
	\begin{enumerate}
		\item  If $(R,+,\circ)$ is an $F$-brace with underlying vector space $R$, then $T_\circ(R)=\{\tau_y:\,y\in R\}$, where $\tau_y:R\longrightarrow R,\, x\mapsto x\circ y$ is a regular subgroup of the affine group $\agl(R)$.
		\item  The map
		\[f: \text{\emph{\texttt{RB}}}\longrightarrow\text{\emph{\texttt{T}}},\, (R,+,\circ)\mapsto T_\circ(R)\]
		is a bijection. Under this correspondence, isomorphism classes of $F$-braces correspond to conjugacy classes under the action of $\gl(R)$ of regular subgroups of $\agl(R)$.
	\end{enumerate}
\end{theorem}
What we have recalled in Theorems~\ref{tca05}, \ref{Rump}, and \ref{thm:CR}  can be summarised as follows.
\begin{remark}
	 Given a vector space $R$ over a field $F$, any of the following data uniquely determines the others:
	\begin{enumerate}
		\item  a regular subgroup $T_{\circ}<\agl(R)$;
		\item an $F$-brace $(R,+,\circ)$;
		\item an associative radical $F$-algebra $(R,+,\cdot)$.
	\end{enumerate}
\end{remark}

It is well known that the specification of a skew brace with additive group $(B,+)$ corresponds to the choice of a regular subgroup of the holomorph of $(B,+)$~\cite{GV16,CARANTI2020647}. Similarly, specifying a bi-skew brace corresponds to selecting a regular subgroup of the holomorph that is normalised by the image of the right regular representation (which we denote here by $T_+$), as stated in the following theorem due to Caranti~\cite[Theorem 3.1]{CARANTI2020647}.

\begin{theorem}{\cite[Theorem 3.1]{CARANTI2020647}}\label{Tcaranti}
	Let $(B,+)$ be a group. The following data are equivalent:
	\begin{enumerate}
		\item  a bi-skew brace $(B,+,\circ)$;
		\item  a regular subgroup $N < \mathrm{Hol}(B,+)$ which is normalised by $T_+(B)$.
	\end{enumerate}
\end{theorem}

\subsection{More equivalences}
The results reviewed in the previous section illustrate the interconnection between the families of regular subgroups $T_{\circ}<\agl(R)$, $F$-braces $(R,+,\circ)$, and associative radical $F$-algebras $(R,+,\cdot)$. We aim to emphasise a similar equivalence when focusing on the family of regular subgroups $T_{\circ}<\agl(R)$ with the dual normalization property, as defined in Sec.\ref{sec:crittoreq}, which triggers the alternative differential attack. It is noteworthy here that the product defined in Eq.\eqref{eq:infoprod}, depicting the discrepancy between the expected difference after the key-addition layer in the case of XOR and $\circ$ differentials (see~Eq.\eqref{eq:keydiff} and Eq.~\eqref{eq:keydiffcirc}), is the product of an $F$-algebra $(R,+,\cdot)$ that we soon characterise.
With this in mind, we recall that Childs gave examples of nontrivial bi-skew braces generated from nilpotent algebras of class two~\cite{Ch19}.

\begin{theorem}{\cite[Proposition 4.1]{Ch19}}\label{Tchilds}
	Let $(R,+,\cdot)$ be a nilpotent $\mathbb{F}_p$-algebra of $\mathbb{F}_p$-dimension $n$. Define the circle operation on $R$ by
	\[x\circ y=x\cdot y+x+y.\]
	Then $(R,+,\circ)$ is a bi-skew brace if and only if $R^3=0$ (i.e.\ for every $x,y,z\in R$, $x\cdot y\cdot z=0$).
\end{theorem}

Stefanello and Trappeniers expanded the previous theorem to apply to radical rings~\cite[Theorem 3.13 and Example 3.16]{ST23}. We present a convenient version of their result below. Later, in Sect.~\ref{sec:algebras}, we demonstrate a version of Theorem~\ref{theoremST23} tailored for a finite-dimensional algebra $R$ over a field $F$, utilising a suitable bilinear map associated with $R$.
\begin{theorem}\label{theoremST23}
	Let $R$ be a radical ring. Then the corresponding two-sided brace is a bi-skew brace if and only if $R^3=0$.
\end{theorem}

We are now ready to illustrate the connection between a specific instance of bi-skew braces, nilpotent algebras of class two with additional specified conditions, and subgroups of the affine group with the dual normalisation property. To accomplish this, we introduce three families of binary structures intended for cryptographic applications, with the field specified as $\mathbb{F}_2$.

\begin{definition}\label{def:families}
	Let $R$ be a finite dimensional vector space over the field with two elements $\mathbb{F}_2$. 
	Let us define
	\begin{enumerate}
		\item the family \bbb\ of \emph{binary bi-braces}, where a binary bi-brace  is a triple $(R,+,\circ)$ such that
		\begin{enumerate}[(a)]
			\item  $(R,\circ)$ is a group,
			\item\label{eq:comm-circ}  $x\circ x=0$,
			\item  $(x+y)\circ z=x\circ z+z+y\circ z$,
			\item  $x\circ y+z=(x+z)\circ z\circ(y+z)$,
			\item  $\delta(x\circ y)=x\circ (\delta y)+(\delta+1)x$,
		\end{enumerate}
		for every $x,y,z\in R$, $\delta\in\mathbb{F}_2$;
		\medskip

		\item  the family \baa\ of \emph{binary alternating algebras}, where a binary alternating algebra is a triple $(R,+,\cdot)$ such that 
		\begin{enumerate}[(a)]
			\item $(R,+,\cdot)$ is an $\mathbb{F}_2$-algebra,
			\item $(x\cdot y)\cdot z=0=x\cdot(y\cdot z)$ (\emph{nilpotency of class two}),
			\item \label {eq:alter}$x\cdot x=0$  (\emph{alternating}),
		\end{enumerate}
		for every $x,y,z\in R$;
		\medskip
		
		\item  the family \ear\ of \emph{elementary abelian regular subgroups}  $T_{\circ}(R)=\{\tau_y : y\in R\}$ of the affine group $\agl(R)$ which are normalised by the translation group $T_+(R)$, i.e.\ $T_+(R)<N_{\sym(R)}(T_\circ(R))$. 
	\end{enumerate}
\end{definition}
Notice that in Definition~\ref{def:families}, (\ref{eq:comm-circ}) implies that $\circ$ is also commutative, indeed
	\[0=(x+y)\circ(x+y)=x\circ x+y\circ x+x+x\circ y+y\circ y+y+x+y=x\circ y+y\circ x.\]
	Similarly, also $\cdot$ is commutative. Indeed
	\[0=(x+y)\cdot(x+y)=x\cdot x+x\cdot y+y\cdot x+y\cdot y=x\cdot y+y\cdot x.\]

\begin{remark}
		If $(R,+,\circ)$ is a binary bi-brace, then $(R,\circ,+)$ is an $\mathbb{F}_2$-brace as in Definition~\ref{def:fbra}. Indeed, by Definition~\ref{def:families}~\eqref{eq:comm-circ}, $(R,\circ)$ is a vector space over $\mathbb{F}_2$. Thus, our task is to confirm that Eq.~\eqref{fbrace} is valid for $(R,\circ,+)$, namely
		\[\delta(x+y)=(x+\delta y)\circ ((\delta+1)x)\]
		for every $x,y\in R$, $\delta\in\mathbb{F}_2$. Observing that $x\circ y=x\cdot y+x+y$, $x\cdot x=0$ and $\delta(\delta+1)=0$, we have
		\begin{align*}
			(x+\delta y)\circ ((\delta+1)x)=\,&x+\delta y+(\delta+1)x+(x+\delta y)\cdot((\delta+1)x)\\
			=\,&x+\delta y+\delta x+x+(\delta+1)(x\cdot x)+\delta(\delta+1)(y\cdot x)\\
			=\,&\delta(x+y).
		\end{align*}
	\end{remark}

\subsection{"New" equivalences}
The families of Definition~\ref{def:families} are \emph{elementwise} equivalent, meaning that when one of these structures is given, the other two can be bijectively derived using the equalities
\[
x\circ y=x+y+x\cdot y=x\tau_y,
\]
where $x,y\in R$. 
Moreover, this correspondence preserves the isomorphism classes of \bbb, the isomorphism classes of \baa, and the conjugacy classes of \ear\ under the action of $\GL(R)$, i.e.\ for every $\beta\in\gl(R)$, the following holds:
\[
(x\beta\circ y\beta)\beta^{-1}=x+y+(x\beta\cdot y\beta)\beta^{-1}=x\beta\tau_{y\beta}\beta^{-1}.
\]
In particular, this shows that the automorphism groups of these structures coincides, meaning:
\begin{equation}\label{eq:auts}
	\aut(R,+,\circ)=\aut(R,+,\cdot)=\gl(R)\cap\aut(R,\circ).
\end{equation}

The mentioned equivalences can be obtained from the results referenced in the following remark.

\begin{remark}\label{rmk:equiv}
The information presented in Theorem~\ref{Tcaranti}, tailored for a finite-dimensional vector space $(R,+)$ over a field $F$, establishes the equivalence of the following data:
	\begin{enumerate}
		\item  a regular subgroup $T_{\circ}(R)=\{\tau_y:\, y\in R\}$ of $\agl(R)$ which is normalised by $T_+(R)$, i.e.\ $T_+(R)<N_{\sym(R)}(T_{\circ}(R))$;
		\item a two-sided bi-skew brace $(R,+,\circ)$ over $F$.
	\end{enumerate}
	In particular, in the case of Definition~\ref{def:families}, this establishes the equivalence between the families {\ear} and {\bbb}. Additionally, from Theorem~\ref{Tchilds}, if $F=\mathbb{F}_p$ the following data are equivalent:
	\begin{enumerate}
		\item  a finite dimensional $\mathbb{F}_p$-algebra of nilpotency class two $(R,+,\circ)$;
		\item a two-sided bi-skew brace $(R,+,\circ)$ over $\mathbb{F}_p$.
	\end{enumerate}
	This provides the equivalence between the families {\baa} and {\bbb}. Finally, if $T_{\circ}(R)$ is abelian then the following data are equivalent~\cite[Lemma 3]{Ca05}, \cite{CARANTI2020647}:
	\begin{enumerate}
		\item  an abelian regular subgroup $T_{\circ}(R)=\{\tau_y:\, y\in R\}$ of $\agl(R)$ which is normalised by $T_+(R)$, i.e.\ $T_+(R)<N_{\sym(R)}(T_{\circ}(R))$;
		\item  a commutative $F$-algebra $(R,+,\cdot)$ of nilpotency class two.
	\end{enumerate}
	The last result proves the equivalence between the families {\ear} and {\baa}.
\end{remark}

\subsection{A convenient representation for \bbb}
In this section, we revisit a convenient method for constructing and representing binary bi-braces and we explore how the cryptographic requirements outlined in Sect.~\ref{sec:crittoreq} are incorporated into the axioms of the algebraic families of Definition~\ref{def:families}.\\

Let us start by recalling that, 
for every $y\in R=\mathbb{F}_2^n$, the map
\[\lambda_y: R\longrightarrow R,\quad x\mapsto x\lambda_y=x\circ y+y\]
is an automorphism of the additive group $(R,+)$~\cite[Lemma 2.6]{Ce18}.
In particular, any binary bi-brace is uniquely determined by the so-called \emph{lambda} map
\[\lambda:(R,\circ)\longrightarrow\gl(R,+),\quad x\mapsto \lambda_x\]
which is a group homomorphism. This implies that, once a basis $e_1,\dots,e_n$ of $R$ is fixed, a binary bi-brace is uniquely determined by the matrices $\lambda_{e_1},\dots,\lambda_{e_n}$.
We also recall that, for every $y\in R$, the set $T_\circ(R)$ consisting of the affine maps
\[\tau_y:\, R\longrightarrow R,\quad x\mapsto x\lambda_y+y\]
belongs to \ear, and vice versa, given an element of \ear, one can define the lambda map of a binary bi-brace.
Under the hypothesis that the kernel of the lambda map
\[\ker(\lambda)=\{y\in R \mid \lambda_y=1\}\] 
is spanned by the last $d$ vectors of the canonical basis $\{e_1,\dots,e_m,e_m+1,\dots,e_{m+d}\}$, the matrices $\lambda_{e_i}$, for $1\leq i\leq m$, can be represented in the following block form
\[\lambda_{e_i}=\pmat{1_m&\Theta_i\\0&1_d}\]\label{lambda}
(see~\cite[Theorem 3.11]{calderini2021properties}). The horizontal concatenation matrix of the $m\times d$ matrices $\Theta_i$
\[\Theta=\pmat{\Theta_1&\dots&\Theta_m}\]
is the \emph{defining matrix} of an operation $\circ$, as introduced by Civino et al.~\cite[Definition 3.5]{CBS19}. The following result, which is revisited later from the perspective of binary alternating algebras, provides a compact way of representing every structure in \bbb\ (and consequently in \ear\ and \baa).

\begin{theorem}{\cite{CBS19}}\label{theta}
	Let $\Theta_{i,j}$ be $m\times d$ matrices over $\mathbb{F}_2$ and $\Theta_{i,j}$ the $d$-dimensional row vector of the matrix $\Theta_i$ for every $1 \leq i,j\leq m$. Then $\Theta=\pmat{\Theta_1&\dots&\Theta_m}$ is the defining matrix of a binary bi-brace such that $\ker(\lambda)=\langle e_{m+1},\dots,e_{m+d}\rangle$ if and only if, for each $1 \leq i,j\leq m$ the following conditions holds:
	\begin{enumerate}
		\item  $\Theta_{i,i}=0$;
		\item  $\Theta_{i,j}=\Theta_{j,i}$;
		\item  $\Theta_1,\dots,\Theta_m$ are linearly independent. 
	\end{enumerate}
	\end{theorem}
	In accordance with the equivalence of the families of Definition~\ref{def:families} as discussed in Remark~\ref{rmk:equiv}, it is important to notice that 
	 $\ker(\lambda)=\{y\in R \mid \lambda_y=1_n\}$ coincides with the set \[\{y\in R \mid \tau_y\in T_+(R)\cap T_{\circ}(R)\}\]
	as well as with the following two ideals: the \emph{socle} of the binary bi-brace $(R,+,\circ)$,
	\[\soc(R)=\{y\in R \mid \forall\, x\in R\quad x\circ y=x+y\}\]
	 and the \emph{annihilator} of the binary alternating algebra $(R,+,\cdot)$,
	 \[\Ann(R)=\{y\in R \mid \forall\, x\in R\quad x\cdot y=0\}.\]
	 
Considering the cryptographic context that prompted this discussion, it is important to recognise that "keys" $k$ from
	 $\soc(R)$ satisfy $x+k=x\circ k$ for each $x \in R$. This implies that, if $k \in \soc(R)$ the probability of Eq.~\eqref{eq:keydiffcirc}, the depicting a $\circ$-difference spreading through the key-addition layer, is one. Consequently, such keys have been referred to as \emph{weak keys} within cryptographic discussions~\cite{CBS19}.
	 Now that the context is made more clear, it is not hard to notice that, in the case of Definition~\ref{def:families}, we can write the output difference to the key-addition layer when considering an input difference equal to $\Delta$ as 
	 
\begin{eqnarray}\nonumber
(x +k) \circ \left((x \circ \Delta)+k\right)&=&(x+k)\p(x+\D + x \cdot \D +k )\\\nonumber
&=& x+k+x+\D+x \cdot \D + k  \\\nonumber
&& + (x+k) \cdot (x+\D+x \cdot \D + k)\\\label{eq:triple}
&=& \D + x \cdot \D + x\cdot x + x \cdot \D + x \cdot x \cdot \D\\\nonumber
&& + x \cdot k + k \cdot x  + k \cdot \D + k \cdot \D \cdot x + k \cdot k\\ \label{eq:nox}
&=& \Delta + k \cdot \D,
\end{eqnarray}
for each $k$, weak or not.
The previous argument highlights the necessity of the dual normalisation properties of groups from \ear, which, due to the equivalence demonstrated in Remark~\ref{rmk:equiv}, translates into a class-two nilpotency requirement on the algebra $(R,+, \cdot)$, ensuring all triple products in Eq.~\eqref{eq:triple} vanish. Notice that obtaining Eq.~\eqref{eq:nox} from Eq.~\eqref{eq:triple} also requires that $(R,+,\cdot)$ is alternating (see~Definition~\ref{def:families}(\ref{eq:alter})), thus implying commutativity as well. 
Moreover, it is important to observe that the nontrivial ideal $R^2=\langle x\cdot y:\, x,y\in R\rangle$ is contained in $\Ann(R)$ due to the nullity of each triple product. This yields a dual insight. Firstly, the predictability of the value of Eq.~\eqref{eq:nox} depends on the magnitude of $R^2$, namely its dimension as a subspace of $(R,+)$. Secondly, since $R^2 \leq \Ann(R)=\soc(R)$, it provides clarity on the location of the \emph{error} $k \cdot \D$.\\

In the upcoming sections, we investigate binary bi-braces from the perspective of class-two nilpotent algebras. We characterise isomorphic algebras within this context, using bilinear maps and congruences of matrix spaces. Through the characterisation of isomorphisms among algebras in the family \baa, we derive the structure of the automorphism group of an algebra $(R,+,\cdot) \in$ \baa. It is worth noting (refer to Eq. \eqref{eq:auts}) that $\text{Aut}(R,+,\cdot) = \text{GL}(R) \cap \text{Aut}(R,\circ)$ represents the group of linear maps respecting both $+$ and $\circ$ operations, enabling deterministic propagation of differentials regardless of the operation under consideration. We further specialise this result to the scenario where $R^2$ has minimal dimension, i.e.\ $\dim(R^2)=1$, particularly in its application to alternative differential attacks.


\section{Nilpotent algebras of class two}\label{sec:algebras}

This section presents a description of the automorphisms and isomorphisms of nilpotent algebras of class two, framed in terms of equivalence between appropriate bilinear maps and congruences among matrix spaces. It is important to highlight that the forthcoming results extend beyond the realm of algebras in \baa, implying that the algebras under consideration need not be alternating or commutative. Subsequently, in Sect.~\ref{sec:binalg}, these findings are adapted to the context of binary alternating algebras.\\

We begin by briefly revisiting the fundamentals of nilpotent algebras of class two and nondegenerate bilinear maps. This is complemented by two examples that offer an initial insight into the relationship between these two concepts and the subsequent result of generalisation.

Let $(R,+,\cdot)$ be an $F$-algebra of nilpotency class two, i.e.\ $R^3=0$ and $R^2\neq0$. 
The nontrivial ideal $R^2=\langle x\cdot y:\, x,y\in R\rangle$ is contained in the ideal 
\[
\Ann(R)=\{x\in R\ |\ \forall\ y\in R\quad x\cdot y=0=y\cdot x\}
\] because each triple product is null, and in particular $\Ann(R)\neq 0$.

\begin{example}\label{ex:ann}
	Let us consider the vector space $R=\mathbb{F}_2^3$ endowed with the following product defined over the canonical basis $\{e_1,e_2,e_3\}$
	\[\begin{array}{c|ccc}
		\cdot&e_1&e_2&e_3\\
		\hline
		e_1	 &0&e_3&0\\
		e_2	 &e_3&0&0\\
		e_3	 &0&0&0
	\end{array}\quad .\]
	Then $(R,+,\cdot)$ is a nilpotent algebra of class two with $\Ann(R)=R^2=\langle e_3\rangle$.
\end{example}

\begin{definition}
	Given two vector spaces $V$ and $W$ over a field $F$, a map \[\phi: V\times V\longrightarrow W\] is \emph{bilinear} if, for every $x_1,x_2,x_3\in V$, $\delta\in F$,
	\begin{itemize}
		\item $\phi(x_1+x_2,x_3)=\phi(x_1,x_3)+\phi(x_2,x_3)$,
		\item $\phi(x_1,x_2+x_3)=\phi(x_1,x_2)+\phi(x_2,x_3)$,
		\item $\delta\phi(x_1,x_2)=\phi(\delta x_1,x_2)=\phi(x_1,\delta x_2)$,
	\end{itemize}
	and \emph{nondegenerate} if 
	\begin{itemize}
		\item $\forall x_2\in V\quad \phi(x_1,x_2)=0=\phi(x_2,x_1)\quad \implies\quad x_1=0$.
	\end{itemize}
\end{definition}

\begin{example}\label{esempio2}
	Let us consider the vector space $R=\mathbb{F}_2^3$ and the following matrix over $\mathbb{F}_2$ \[B=\pmat{0&1\\1&0}.\]
	The map $\phi:\, \langle e_1,e_2\rangle\times\langle e_1,e_2\rangle\longrightarrow \langle e_3\rangle$ defined by
	\[\phi\big((x_1,x_2,0),(y_1,y_2,0)\big)=\big(0,0,(x_1,x_2)\,B\,(y_1,y_2)^t\big)\]
	is a nondegenerate bilinear map.
	Now, let us consider the matrix 
	\[\widehat{B}=\pmat{0&1&0\\1&0&0\\0&0&0}\]
	and the map $\widehat{\phi}:R\times R\longrightarrow\langle e_3\rangle$ defined by
	\[\widehat{\phi}\big((x_1,x_2,x_3),(y_1,y_2,x_3)\big)=\big(0,0,(x_1,x_2,x_3)\,\widehat{B}\,(y_1,y_2,y_3)^t\big).\]
	The evaluation of $\widehat{\phi}$ over the canonical basis $\{e_1,e_2,e_3\}$
	\[\begin{array}{c|ccc}
		\phi&e_1&e_2&e_3\\
		\hline
		e_1	 &0&e_3&0\\
		e_2	 &e_3&0&0\\
		e_3	 &0&0&0
	\end{array}\]
	gives the algebra product of Example~\ref{ex:ann}. 
	Observe that $\phi$ is the restriction of the map $\widehat{\phi}$, defined over the entire space $R$, to a complement of $\Ann(R)=\langle e_3\rangle$.
\end{example}

Building upon the previous two examples, in a more general context, it becomes evident that: the product of a nilpotent algebra of class two induces a nondegenerate bilinear map, defined on a complement of the annihilator and taking values in the annihilator itself. Conversely, given a nondegenerate bilinear map $\phi:V\times V\longrightarrow W$ over $F$-vector spaces $V$ and $W$, it is possible to define a product on the direct sum $V\oplus W$ that equips it with the structure of a nilpotent algebra of class two.
The following result shows this connection between nilpotent algebras of class two and nondegenerate bilinear maps.

\begin{theorem}\label{th1}
	Let $(R,+,\cdot)$ be a nilpotent algebra of class two over a field $F$, let $W=\Ann(R)$ and $V$ be a complement of $W$. Then 
	\[\widehat{\phi}:R\times R\longrightarrow W,\quad \widehat{\phi}(x,y)=x\cdot y\]
	is a bilinear map and its restriction to $V$, $\phi:\,V\times V\longrightarrow W$, is a nondegenerate bilinear map, i.e.\ given $x\in V$, if ${\phi}(x,y)=0={\phi}(y,x)$ for each $y\in V$, then $x=0$.
	
	Conversely, given two $F$-vector spaces $V$ and $W$ and a nondegenerate bilinear map $\phi: V\times V\longrightarrow W$, the triple $(V\oplus W,+,\cdot)$ is an associative $F$-algebra of nilpotency class two with annihilator $0\oplus W$ where
	\begin{align}\label{product}
		(x_1,y_1)\cdot(x_2,y_2)=(0,\phi(x_1,x_2))
	\end{align}
	for every $x_1,x_2\in V$, $y_1,y_2\in W$.
	\begin{proof}
		The bilinearity of $\widehat{\phi}:R\times R\longrightarrow W$ follows from the left/right-distributivity and compatibility with scalars properties of the algebra $R$. Moreover, an element $x\in V$ such that $x\cdot y=0=y\cdot x$ for each $y\in V$ belongs to $\Ann(R)\cap V=0$.

Let $\phi: V\times V\longrightarrow W$ be a nondegenerate bilinear map. The product defined by Eq.~\eqref{product} satisfies the right-distributivity (and similarly the left-distributivity)
		\begin{align*}
			(x_1+x_2,y_1+y_2)\cdot(x_3,y_3)=\,&(0,\phi(x_1+x_2,x_3))\\
			=\,&(0,\phi(x_1,x_3))+(0,\phi(x_2,x_3))\\
			=\,&(x_1,y_1)\cdot(x_3,y_3)+(x_2,y_2)\cdot(x_3,y_3),
		\end{align*}
		and the compatibility with scalars
		\begin{align*}
			\delta((x_1,y_1)\cdot (x_2,y_2))=(0,\delta\phi(x_1,x_2))=\,&(0,\phi(\delta x_1,x_2))=(\delta x_1,\delta y_1)\cdot (x_2,y_2)\\
			=\,&(0,\phi(x_1,\delta x_2))=(x_1,y_1)\cdot(\delta x_2,\delta y_2).
		\end{align*}
		Moreover, associativity and nilpotency class two follow at once from
		\[(x_1,y_1) \cdot \big((x_2,y_2) \cdot(x_3,y_3) \big)\ =\ (x_1,y_1) \cdot (0,\phi(x_2,x_3))\ =\ (0,\phi(x_1,0))\ =\ 0\]
		and from 
		\[\big((x_1,y_1)\, \cdot\, (x_2,y_2)\big) \cdot(x_3,y_3)\ =\ 0.\]
		By definition, $\Ann(V\oplus W,+,\cdot)$ is the set
		\[\left\{(x_1,y_1)\in V\oplus W \mid \forall\ (x_2,y_2)\in V\oplus W\quad (0,\phi(x_1,x_2))=(0,0)=(0,\phi(x_2,x_1))\right\}\]
		which coincides with $\left\{(0,y_1):\, y_1\in W\right\}=0\oplus W$ by the nondegeneracy of $\phi$.
	\end{proof}
\end{theorem}

Considering the theorem just established, it is logical to propose the following definition.

\begin{definition}
	Given a nilpotent $F$-algebra of class two $(R,+,\cdot)$, $W=\Ann(R)$ and $V$ a complement of $W$, we define the restriction $\phi:V\times V\longrightarrow W$ of the map
	\[\widehat{\phi}: R\times R\longrightarrow W,\quad \phi(x,y)=x\cdot y\]
	as the \emph{bilinear map associated with $R$ relative to $(V,W)$}.
	If $R$ is finite-dimensional, given a basis $v_1,\dots,v_m$ of $V$ and a basis $w_1,\dots,w_d$ of $W$, we define $\phi$ as the \emph{bilinear map associated with $R$ relative to the ordered basis $(v_1,\dots,v_m,w_1,\dots,w_d)$ of $R$}. If bases for $V$ and $W$ are not specified, the canonical bases $F^m=\langle e_1,\dots,e_m\rangle$ and $F^d=\langle e_1,\dots,e_d\rangle$ are assumed.
\end{definition}

Using the concepts introduced above, we prove the following version of Theorem~\ref{theoremST23}~(\cite{ST23a}) on the bijective correspondence between $F$-algebras of nilpotency class two $(R,+,\cdot)$ and two-sided bi-skew braces $(R,+,\circ)$ over $F$.
\begin{theorem}
	Let $F$ be a field and $(R,+,\cdot)$ be an $F$-algebra of nilpotency class two. Then $(R,+,\circ)$ is a two-sided bi-skew brace over $F$ where $x\circ y=x+y+x\cdot y$.
	
	Conversely, if $(R,+,\circ)$ is a two-sided bi-skew brace with underlying $F$-vector space $(R,+)$, then $(R,+,\cdot)$ is a finite dimensional $F$-algebra of nilpotency class two where $x\cdot y=x\circ y-x-y$.
\end{theorem}
\begin{proof}
	Since $(R,+,\cdot)$ is nilpotent, by Theorem \ref{Rump} $(R,+,\circ)$ is a two-sided $F$-brace. Let us consider three elements $x=(x_1,x_2),y=(y_1,y_2),z=(z_1,z_2)\in R=V\oplus W$, where $W=\Ann(R)$ and $V$ is a complement of $W$, and let $\phi:\,V\times V\longrightarrow W$ be the bilinear map associated with $R$. To prove that $(R,\circ,+)$ is a skew brace, we need to show the skew brace-like distributivity:
	\begin{align}\label{fdsb}
		x+(y\circ z)=(x+y)\circ x^{-1}\circ (x+z)\quad \text{for }\, x,y,z\in R.
	\end{align}
	Writing the circle operation as \[(x_1,x_2)\circ(y_1,y_2)=(x_1+y_1,x_2+y_2+\phi(x_1,y_1)),\]
	we obtain \[(x_1,x_2)^{-1}=(-x_1,-x_2+\phi(x_1,x_1)).\]
	Starting from the RHS (right-hand side) of Eq.~\eqref{fdsb} we have
	\begin{align*}
		\mathrm{RHS} \text { of Eq.~\eqref{fdsb}}=&(x_1+y_1,x_2+y_2)\circ (-x_1,-x_2+\phi(x_1,x_1))\circ(x_1+z_1,x_2+z_2)\\
		=&(y_1,y_2+\phi(x_1,x_1)+\phi(x_1+y_1,-x_1))\circ(x_1+z_1,x_2+z_2)\\
		=&(y_1,y_2+\phi(y_1,-x_1))\circ(x_1+z_1,x_2+z_2)\\
		=&(x_1+y_1+z_1,x_2+y_2+z_2+\phi(y_1,-x_1)+\phi(y_1,x_1+z_1))\\
		=&(x_1+y_1+z_1,x_2+y_2+z_2+\phi(y_1,z_1)),
	\end{align*}
	where the last expression coincides with the LHS of Eq.~\eqref{fdsb}. Therefore $(R,\circ,+)$ is a (two-sided) skew-brace over $F$.
	
	Conversely, if $(R,+,\circ)$ is a two-sided bi-skew brace over $F$, then $(R,+,\cdot)$ is an $F$-algebra. Given that $(R,\circ,+)$ is a skew brace, we have that Eq.~\eqref{fdsb} holds, therefore
	\[x+y\cdot z+y+z=((x+y)\cdot x^{-1}+x+y+x^{-1})\cdot(x+z)+(x+y)\cdot x^{-1}+x+y+x^{-1}+x+z\]
	for every $x,y,z\in R$. Notice that
	\[x\cdot x^{-1}=-x-x^{-1}=-x^{-1}-x=x^{-1}\cdot x.\]
	Consequently
	\begin{align*}
		0=\,&((x+y)\cdot x^{-1}+x+y+x^{-1})\cdot(x+z)+(x+y)\cdot x^{-1}-x\cdot x^{-1}-y\cdot z\\
		=\,&((x+y)\cdot x^{-1}+x+y+x^{-1})\cdot(x+z)+y\cdot x^{-1}-y\cdot z\\
		=\,&(x+y)\cdot x^{-1}\cdot(x+z)+(x+y+x^{-1})\cdot(x+z)+y\cdot x^{-1}-y\cdot z\\
		=\,&x\cdot x^{-1}\cdot x+x\cdot x^{-1}\cdot z+y\cdot x^{-1}\cdot x+y\cdot x^{-1}\cdot z+\\
		\qquad&+x^{-1}\cdot x+x^{-1}\cdot z+x\cdot x+x\cdot z+y\cdot x+\cancel{y\cdot z}+y\cdot x^{-1}-\cancel{y\cdot z}  \\
		=\,&-\cancel{x\cdot x}-\cancel{x^{-1}\cdot x}-\cancel{x\cdot z}-\cancel{x^{-1}\cdot z}-\cancel{y\cdot x^{-1}}-\cancel{y\cdot x}+y\cdot x^{-1}\cdot z+\\
		\qquad&+\cancel{x^{-1}\cdot x}+\cancel{x^{-1}\cdot z}+\cancel{x\cdot x}+\cancel{x\cdot z}+\cancel{y\cdot} x+\cancel{y\cdot x^{-1}}  \\
		=\,&y\cdot x^{-1}\cdot z.
	\end{align*}
	This proves that $R^3=0$, meaning that $(R,+,\cdot)$ is nilpotent of class two.
\end{proof}

Building upon Example~\ref{esempio2}, we proceed to introduce a series of matrices linked with a nondegenerate bilinear map. The vector space formed by these matrices holds significant importance in detailing the bijective morphisms of the algebras.

\subsection{Introducing matrix spaces}	
Let $(R,+,\cdot)$ be a nilpotent $F$-algebra of class two and let $\phi:\,V\times V\longrightarrow W$ be the bilinear map associated with $R$ relative to the ordered basis $(v_1,\dots,v_m,w_1,\dots,w_d)$. Then the bilinear map defined over the entire space $\widehat{\phi}:\, R\times R\longrightarrow W$ is determined by a sequence $(\widehat{B}_1,\dots,\widehat{B}_d)$ of $d$ square $(m+d)\times (m+d)$ matrices over $F$. More precisely, we can write
		\[\widehat{\phi}(x,y)=\phi_1(x,y)w_1+\dots+\phi_d(x,y)w_d\]
		where 
		\[\phi_k: R\times R\longrightarrow F,\quad \phi_k(x,y)=x\widehat{B}_ky^t\]
		is a bilinear form over $F$ for $1\leq k\leq d$. Conversely, if $\phi_1,\dots,\phi_d$ are given then the matrix $\widehat{B}_k$ is defined by
		\[\widehat{B}_k=\left[\begin{array}{c|c}\phi_k(v_i,v_j)&0\\\hline 0&0\end{array}\right],\quad 1 \leq i,j \leq m.\]
		We denote by $B_k$ the $m\times m$ submatrix of $\widehat{B}_k$ whose $(i,j)$ entry $B_k[i,j]$ is given by $v_i\widehat{B}_kv_j^t$ for each $1 \leq i,j \leq m$ and $1 \leq k \leq d$. Then the matrices $B_k$ are the $m\times m$ matrices associated to the bilinear map $\phi:V\times V\longrightarrow W$.
		
Conversely, starting from a sequence $B_1,\dots,B_d$ of square matrices of size $m$ over a field $F$, one can define a nilpotent algebra of class two by endowing the vector space $F^m\oplus F^d$ with the product (see~Theorem~\ref{th1}) \[(x_1,y_1)\cdot(x_2,y_2)=(0,(x_1B_1x_2^t,\dots,x_1B_dx_2^t)).\]
Ensuring the nondegeneracy of the bilinear map $\phi(x_1,x_2)=(x_1B_1x_2^t,\dots,x_1B_dx_2^t)$ requires only that the matrix formed by horizontally concatenating $B_1,...,B_d$ has a rank of $m$, as noted below.

\begin{remark}
A sequence $B_1,\dots,B_d$ of $m\times m$ matrices over a field $F$ defines a nondegenerate bilinear map $\phi:F^m\times F^m\longrightarrow F^d$ if and only if $\rank\pmat{B_1&\dots&B_d}=m$.

			In fact, we observe that a vector $x=x_1e_1+\dots+x_me_m\in F^m$ yields a null linear combination of the $m$ row vectors $(e_iB_1,\dots,e_iB_d)=\left((e_iB_1e_1^t,\dots,e_iB_1e_m^t),\dots,(e_iB_de_1^t,\dots,e_iB_de_m^t)\right)$, i.e.
			\[\underbrace{(0_m,\dots,0_m)}_{d-\text{components}}=\sum_{i=1}^m{x_i(e_iB_1,\dots,e_iB_d)}=\left((xB_1e_1^t,\dots,xB_1e_m^t),\dots,(xB_de_1^t,\dots,xB_de_m^t)\right),\]
			if and only if $(xB_he_1^t,\dots,xB_he_m^t)=(0,\dots,0)$ for each $1\leq h\leq d$, i.e.
			\[(0,\dots,0)=(xB_1e_j^t,\dots,xB_de_j^t)=\phi(x,e_j)\quad\forall\,1\leq j\leq m.\]
			Thus $\phi(x,e_j)=0$ for every $1\leq j\leq m$ if and only if $x=0$ is the unique vector which induces a null linear combination of the $m$ rows of $\pmat{B_1&\dots&B_d}$. Analogously, $\phi(e_j,x)=0$ for every $1\leq j\leq m$ if and only if $x=0$ is the unique vector which induce a null linear combination of the $m$ columns of $\pmat{B_1^t&\dots&B_d^t}^t$. Therefore, $\phi$ is nondegenerate if and only if $\pmat{B_1&\dots&B_d}$ has a rank of $m$.
		\end{remark}
		
\begin{definition}
Let $V=\langle v_1,\dots, v_m\rangle$ and $W=\langle w_1,\dots,w_d\rangle$ be two finite-dimensional vector spaces over a field $F$, and let $R=V\oplus W$. Given a nondegenerate bilinear map $\phi:V\times V\longrightarrow W$, we denote by $\cdot_\phi$ the algebra product induced by $\phi$ over $R$ as defined in Eq.~\eqref{product}. The matrices $B_1,\dots,B_d$ satisfying
\[\phi(x_1,x_2)=(x_1B_1x_2^t,\dots,x_1B_dx_2^t)\quad \text{for } x_1,x_2\in V\]
are referred to as the \emph{defining matrices of $(R,+,\cdot_\phi)$}.
\end{definition}

The subsequent step involves identifying a suitable characterisation of the ideal $R^2$ using the defining matrices $B_1, \dots, B_d$.
Notice that 
\[R^2=\left\langle(x_1,y_1)\cdot_\phi(x_2,y_2):\, x_1,x_2\in V,\,y_1,y_2\in W\right\rangle=\langle (0,\phi(x_1,x_2)):\, x_1,x_2\in V\rangle\]
is a vector subspace of $\Ann(R)=W$ generated by the set \[\{(0,\phi(v_i,v_j)): 1\leq i,j\leq m\}.\]
Moreover, there exist coefficients $v_{ij}^k\in F$ such that
\[\phi(v_i,v_j)=\sum_{k=1}^{d}{v_{ij}^kw_k}=\sum_{k=1}^{d}{B_k[i,j]w_k}=\sum_{k=1}^d{(v_iB_kv_j^t)w_k},\]
where $v_{ij}^k$ is the $(i,j)$ entry of the matrix $B_k$, i.e.\ $v_{ij}^k=B_k[i,j]=v_iB_kv_j^t$. Specifically, we obtain
\[R^2\simeq\langle(B_1[i,j],\dots,B_d[i,j]):\, 1\leq i,j\leq m\rangle.\]
The next proposition directly follows from the aforementioned isomorphism.
\begin{proposition}\label{R2}
	Let $(R,+,\cdot_\phi)$ be an algebra of nilpotency class two over a field $F$, $R=F^m\oplus F^d$, and let $B_1,\dots,B_d$ be its defining matrices. Then $\dim R^2=\dim\langle B_1,\dots,B_d\rangle$.
\end{proposition}
\begin{proof}
	Let us consider the canonical basis $\left\{E_{11},E_{12}\dots,E_{hk},\dots,E_{mm}\right\}$ of the $(m\times m)$-matrix vector space over $F$, i.e.\ for $1\leq h,k\leq m$, let
	\[E_{hk}[i,j]=\left\{\mat{1&\text{if }(h,k)=(i,j)\\0&\text{otherwise}}\right. .\]
	Then, for each $1\leq l\leq d$, it is uniquely determined an $m^2$-tuple of coefficients \[(c_{11}^l,c_{12}^l,\dots,c_{hk}^l,\dots,c_{mm}^l)\in F^{m^2}\]
	such that
	\[B_l=\sum_{1\leq h,k\leq m}{c_{hk}^lE_{hk}}.\]
	This means that $(B_1[i,j],\dots,B_d[i,j])=(c_{ij}^1,\dots,c_{ij}^d)$ for each $1\leq i,j\leq m$. 
	Let us consider the $m^2\times d$ coefficient matrix denoted by \[C=\pmat{c_{11}^1&\dots&c_{11}^d\\c_{12}^1&\dots&c_{12}^d\\\vdots&&\vdots\\c_{mm}^1&\dots&c_{mm}^d}.\]
The assertion is derived from the observation that
	\[\dim_F\langle B_1,\dots,B_d\rangle=\rank(C)=\dim\left\langle(c_{ij}^1,\dots,c_{ij}^d):\, 1\leq i,j,\leq m\right\rangle=\dim_FR^2.\]
\end{proof}

\begin{example}\label{example1}
	Let us consider the following bilinear map
	\begin{align*}
		\phi:\mathbb{F}_2^4\times\mathbb{F}_2^4\longrightarrow\mathbb{F}_2^4,\quad \phi(x_1,x_2)=(x_1B_1x_2^t,x_1B_2x_2^t,x_1B_3x_2^t,x_1B_4x_2^t),
	\end{align*}
	where
	\[B_1=\pmat{0&0&1&1\\
		0&0&1&0\\
		1&1&0&1\\
		1&0&1&0},\ B_2=\pmat{0&0&0&0\\
		0&0&1&1\\
		0&1&0&0\\
		0&1&0&0},\ B_3=\pmat{0&0&0&1\\
		0&0&1&1\\
		0&1&0&0\\
		1&1&0&0},\ B_4=\pmat{0&1&1&1\\
		1&0&1&1\\
		1&1&0&1\\
		1&1&1&0}.\]
		
		Note that $\rank(B_1)=4$, which implies that the horizontal concatenation is also of rank four, ensuring that $\phi$ is nondegenerate. Then, $R=\mathbb{F}_2^8\simeq\mathbb{F}_2^4\oplus\mathbb{F}_2^4$ endowed with the operation defined by
		\[x\cdot y=\left(0,0,0,0,x\pmat{B_1&0\\0&0}y^t,x\pmat{B_2&0\\0&0}y^t,x\pmat{B_3&0\\0&0},x\pmat{B_4&0\\0&0}y^t\right)\]
		is a nilpotent algebra of class two with annihilator spanned by the last four vector of the canonical basis $\{e_1,e_2,\dots,e_8\}$. Since the matrices $B_i$ are linearly independent, the space $R^2=\left\langle(0,\phi(x_1,x_2): x_1,x_2\in\mathbb{F}_2^4)\right\rangle$ coincide with the annihilator. It is important to recognise that in general, the image of a bilinear map is not necessarily a vector space. In fact, when computing all possible dot products, we find that $\#\left\{x\cdot y:x,y\in R\right\}=\#\left\{\phi(x_1,x_2): x_1,x_2\in \mathbb{F}_2^4\right\}=15$.
		
	As an illustration of an algebra $(S,+,\cdot)$ where $S^2$ is properly included in $\Ann(S)$, we can examine the dot product defined by the following bilinear map
		\[\psi(x_1,x_2)=\left(x_1B_1x_2^t,x_1B_2x_2^t,x_1B_3x_2^t,x_1(B_1+B_2+B_3)x_2^t\right)\]
		and by the sequence of defining matrices $(B_1,B_2,B_3,B_1+B_2+B_3)$. Clearly, $\Ann(S)=\langle e_5,e_6,e_7,e_8\rangle$. By the  previous proposition, $S^2$ is a space of dimension 3 over $\mathbb{F}_2$ and we can easily determine its generators by writing
		\begin{align*}
			(B_1,B_2,B_3,B_1+B_2+B_3)=(B_1,0,0,B_1)+(0,B_2,0,B_2)+(0,0,B_3,B_3).
		\end{align*}
		Hence,  $\left\{(1,0,0,1),(0,1,0,1),(0,0,1,1)\right\}$ is a system of generator of $\left\langle\psi(x_1,x_2):\,x_1,x_2\in \mathbb{F}_2^4\right\rangle$ and 
		$e_5+e_8=(0,0,0,0,1,0,0,1)$, $e_6+e_8=(0,0,0,0,0,1,0,1)$, $e_7+e_8=(0,0,0,0,0,0,1,1)$ constitute a system of generators of $S^2\simeq\langle(0,\psi(x_1,x_2)):\, x_1,x_2\in \mathbb{F}_2^4\rangle$.
\end{example}

The concepts of equivalence between bilinear maps and congruence between matrix spaces presented below enable us to provide a characterisation of the isomorphisms of nilpotent algebras of class two and a description of their automorphism groups.

\begin{definition}\label{defequiv}
		Two bilinear maps $\phi,\psi: F^m\times F^m\longrightarrow F^d$ are \emph{equivalent} if 
		\begin{align}
			\forall\ x_1,x_2\in F^m\quad \phi(x_1,x_2)D=\psi(x_1A,x_2A)
		\end{align}
		for some $A\in GL(m,F)$, $D\in GL(d,F)$, i.e.\ the following diagram commutes:
		\[\begin{tikzcd}
		F^m\times F^m \arrow[r, "\phi"] \arrow[d, "A\times A"]
		& F^d \arrow[d, "D"] \\
		F^m\times F^m \arrow[r, "\psi"]
		& F^d
		\end{tikzcd}\quad .\]
		
		Two matrix spaces $\langle B_i\rangle_i$ and $\langle C_j\rangle_j$ are \emph{congruent} if there exist an invertible matrix $A$ such that
		\[A\langle C_j\rangle_jA^t=\langle B_i\rangle_i,\]
		i.e.\ $AC_jA^t\in\langle B_i\rangle_i$ for each $j$.
	\end{definition}

We are now ready to express the main result of this section, which characterises the isomorphisms of nilpotent algebras of class two as equivalences between the associated nondegenerate bilinear maps and congruences between the spaces of their defining matrices.

\begin{theorem}\label{congruent}
Let $(R,+,\cdot_\phi)$ and $(S,+,\cdot_\psi)$ be two nilpotent algebras of class two over a field $F$, with underlying vector space $F^m\oplus F^d$ and defining matrices $B_1,\dots,B_d$ and $C_1,\dots, C_d$ respectively. The following statements are equivalent:
		\begin{enumerate}[(i)]
			\item \label{congruent1}  $R$ and $S$ are isomorphic algebras;
			\item \label{congruent2} $\phi$ and $\psi$ are equivalent;
			\item \label{congruent3} $\langle B_1,\dots,B_d\rangle$ and $\langle C_1,\dots, C_d\rangle$ are congruent.
		\end{enumerate}
\end{theorem}
		\begin{proof}
		Let us first prove that \eqref{congruent1} $\implies$ \eqref{congruent2} $\implies$ \eqref{congruent3}.
			Let $f:R\longrightarrow S$,
			 \[ f=\pmat{A&C\\C'&D}\in \gl(F^m\oplus F^d),\] be an isomorphism of algebras. Then, for each $(x_1,y_1),(x_2,y_2)\in F^m\oplus F^d$ the following equations are equivalent
			\begin{align*}
				\big((x_1,y_1)\cdot_\phi(x_2,y_2)\big)f&=(x_1,y_1)f\cdot_\psi(x_2,y_2)f,\\
				(0,\phi(x_1,x_2))f&=(0,\psi(x_1f,x_2f)),\\
				(\phi(x_1,x_2)C',\phi(x_1,x_2)D)&=(0,\psi(x_1A+y_1C',x_2A+y_2C')).
			\end{align*}
			In particular, the last one implies that $C'=0$ and so $A,D$ are invertible matrices such that
			\[\phi(x_1,x_2)D=\psi(x_1A,x_2A),\]
			i.e.\ $\phi$ and $\psi$ are equivalent. From this it follows that
			\[\left(\sum_i{x_1B_ix_2^tD[i,1]},\dots,\sum_i{x_1B_ix_2^tD[i,d]}\right)=(x_1AC_1A^tx_2^t,\dots,x_1AC_dA^tx_2^t),\]
			and so for each $1 \leq j \leq d$
			\begin{align}\label{eq2}AC_jA^t=\sum_i{B_iD[i,j]}.\end{align}
			Therefore $AC_jA^t\in\big\langle B_1,\dots, B_d\big\rangle$, i.e.\ $\langle B_i\rangle_i$ and $\langle C_i\rangle_i$ are congruent.\\ 
			
			Let us now show that \eqref{congruent3} $\implies$ \eqref{congruent2} $\implies$ \eqref{congruent1}.
			If the spaces generated by the defining matrices of $R$ and $S$ are congruent, then there exist a matrix $A\in \gl(m,F)$ and coefficients $D[i,j]\in F$ such that Eq.~\eqref{eq2} holds. Since the map
			\[\varphi_A:\langle C_i\rangle_i\longrightarrow\langle B_i\rangle_i,\quad C_i\mapsto AC_iA^t\]
			is an isomorphism of vector spaces, the coefficient matrix $D$ is invertible. This implies that $A$ and $D$ are invertible matrices that establish the equivalence of the bilinear maps $\phi$ and $\psi$. Consequently, for every $(x_1,y_1),(x_2,y_2)\in F^m\oplus F^d$:
			\begin{align*}
				\big((x_1,y_1)\cdot_\phi(x_2,y_2)\big)\pmat{A&0\\0&D}= &(0,\phi(x_1,x_2)D)\\
				= &(0,\psi(x_1A,x_2A))=(x_1,y_1)\pmat{A&0\\0&D}\cdot_\psi(x_2,y_2)\pmat{A&0\\0&D}.
			\end{align*}
			Thus $\pmat{A&0\\0&D}$ is an isomorphism of algebras.
\end{proof}

\begin{example}
Let us examine the two algebras from Example \ref{example1}, characterised by the defining matrices $(B_1,B_2,B_3,B_4)$ and $(B_1,B_2,B_3,B_1+B_2+B_3)$. Due to the discrepancy in the dimensions of the matrix spaces generated by these sequences (4 and 3, respectively), they are not congruent, and consequently, the algebras are not isomorphic.
\end{example}

Considering that the dimension of the matrix space aligns with the dimension of the $R^2$ space (as indicated in Proposition~\ref{R2}), it is evident that two algebras with $R^2$ spaces of differing dimensions are not isomorphic. Conversely, if $\beta\in\aut(R,+,\cdot_\phi)$, then $R^2\beta=R^2$.
However, there exist nonisomorphic algebras with $R^2$ space of the same dimension, as shown in the following example.

\begin{example}
Let us consider the two algebras with underlying vector space $\mathbb{F}_2^4\oplus\mathbb{F}_2^4$ and defined by the following sequences of defining matrices over $\mathbb{F}_2$:
\[B_1=\pmat{0&1&0&0\\1&0&0&0\\0&0&0&0\\0&0&0&0},\quad B_2=\pmat{0&0&0&0\\0&0&0&0\\0&0&0&1\\0&0&1&0},\quad B_3=B_1+B_2,\quad B_4=0,\]
\[C_1=\pmat{0&0&1&0\\
	0&0&1&1\\
	1&1&0&0\\
	0&1&0&0},\quad
C_2=\pmat{0&0&1&1\\
	0&0&0&1\\
	1&0&0&0\\
	1&1&0&0},\quad C_3=C_1+C_2,\quad C_4=0.\]
Note that $\rank(C_1)=\rank(C_2)=\rank(C_1+C_2)=4$.
However, for each $A\in\gl(4,\mathbb{F}_2)$, $\rank(AB_1A^t)=\rank(B_1)=2$, which means that $AB_1A^t\notin\langle C_i\rangle_i$. Consequently, the matrix spaces $\langle B_i\rangle_i$ and $\langle C_i\rangle_i$ are not congruent, leading to the conclusion that the algebras characterised by these sequences of defining matrices are not isomorphic.
\end{example}

The previous characterisation of isomorphisms provided in Theorem~\ref{congruent} also offers a description the automorphisms of a nilpotent algebra of class two, using the reflexivity of the equivalence between bilinear maps and the congruence between matrix spaces. Specifically, we say that the matrices $A$ and $D$ establish the \emph{self-equivalence} of a bilinear map when the maps $\phi$ and $\psi$ from Definition \ref{defequiv} coincide. Similarly, $A$ and $D$ establish the \emph{self-congruence} when the matrix spaces $\langle B_i\rangle_i$ and $\langle C_i\rangle_i$ from Definition~\ref{defequiv} coincide.

\begin{corollary}\label{cor:aut}
Let $(R,+,\cdot_\phi)$ be a nilpotent algebra of class two over a field $F$, and let $B_1,\dots,B_d$ be its defining matrices. Then the automorphism group $\aut(R,+,\cdot_\phi)$ consists of the matrices of the form 
\[\pmat{A&C\\0&D}\in \gl(R)\]
which establish the self-equivalence of the bilinear map $\phi$ or, equivalently, the self-congruence of the space generated by its defining matrices.
\end{corollary}
\begin{proof}
	The proof follows immediately from the previous theorem by considering $\phi=\psi$, $B_i=C_i$ for each $1 \leq i \leq d$.
\end{proof}
			
We conclude this section with the following result, which specialises the description of the automorphism group of a nilpotent algebra of class two to the case where the ideal $R^2$ is uni-dimensional, or equivalently, the space of defining matrices is uni-dimensional, denoted by $B_1,\dots,B_d\in\langle B\rangle$. In this scenario, the condition that two matrices $A$ and $D$ establish the self-congruence of the matrix space $\langle B\rangle$ translates into two equivalent conditions: $A$ must be an isometry of $B$, and $D$ must fix the generator of $R^2$. It is important to note that this case is particularly relevant in cryptographic applications, owing to the attacker's capability to control the size of the error in the alternative differential attack.
	
\begin{theorem}\label{uni}
		Let $(R,+,\cdot_\phi)$ be an algebra of nilpotency class two over a field $F$, let $R=F^m\oplus F^d$, with uni-dimensional space $R^2=\langle(0,b)\rangle$ and defining matrices $B_1,\dots,B_d\in\langle B\rangle$.
		Then, $\rank(B)=m$ and \[\pmat{A&C\\0&D }\in \gl(F^m\oplus F^d)\] is an automorphism of $R$ if and only if
			$ABA^t=B$ and $bD=b$.
			\end{theorem}
		\begin{proof}
			Since the concatenation matrix $\pmat{B_1&\dots&B_d}$ has rank $m$ and $B_i\in\langle B\rangle$, then $\rank(B)=m$. Let $A$ and $D$ be two invertible matrices of size $m$ and $d$ respectively. Recall that \[R^2=\left\langle (0,\phi(x_1,x_2)):\ x_1,x_2\in F^m\right\rangle=\left\langle (0,b)\right\rangle.\]
			 Clearly $ABA=B$ if and only if, for each $x_1,x_2\in F^m$ and $1 \leq i \leq d$ we have \[x_1AB_iA^tx_2^t=x_1B_ix_2^t,\]
			i.e.\ $\phi(x_1A,x_2A)=\phi(x_1,x_2)$. On the other hand, $bD=b$ if and only if $\phi(x_1,x_2)D=\phi(x_1,x_2)$. Therefore, the two equation holds if and only if $A$ and $D$ establish the self-equivalence of $\phi$.
		\end{proof}
		
		Assume that $(R_1,+,\cdot_\phi)$ and $(R_2,+,\cdot_\psi)$ are two nilpotent algebras of class two with uni-dimensional spaces $R_1^2=\langle(0,b_1)\rangle$ and $R_2^2=\langle(0,b_2)\rangle$, respectively, and spaces of defining matrices $\langle B_1\rangle$ and $\langle B_2\rangle$.
		It is worth noting that if $R_1$ and $R_2$ are isomorphic, then the set of isomorphisms is determined by the matrices $A$ and $D$ such that $AB_2A^t=B_1$ and $b_1D=b_2$.


\section{Binary alternating algebras}\label{sec:binalg}
	
In this section, we specialise the description  of the automorphisms (see~Corollary~\ref{cor:aut} and Theorem~\ref{uni}) of nilpotent algebras of class two to the case of cryptographic interest. Therefore, from now on we consider finite dimensional algebras over the field with two elements $\mathbb{F}_2$ satisfying the following polynomial identities:
	\begin{itemize}
		\item $x\cdot y\cdot z=0$ (nilpotency class two);
		\item $x\cdot x=0$ (nil index two).
	\end{itemize}
We have already noted after Definition~\ref{def:families} that from the previous assumptions $(R,+,\cdot)$ is also commutative, and  we called the class of such algebras as the class \baa\ of binary alternating algebras over $R$.
As a first step, we identify a suitable matrix space for the defining matrices of such algebras.
	
	\begin{proposition}
		The defining matrices $B_1,\dots,B_d$ of a binary alternating algebra $(\mathbb{F}_2^m\oplus\mathbb{F}_2^d,+,\cdot_\phi)$ are symmetric and zero-diagonal.
		\end{proposition}
		\begin{proof}
			For every $(x_1,y_1),(x_2,y_2)\in\mathbb{F}_2^m\oplus\mathbb{F}_2^d$ we have
			\[(0,0)=(x_1,y_1)\cdot(x_1,y_1)=(0,\phi(x_1,x_1))=(0,(x_1B_1x_1^t,\dots,x_1B_dx_1^t))\]
			and
			\[(x_1,y_1)\cdot(x_2,y_2)=(0,(x_1B_1x_2^t,\dots,x_1B_dx_2^t))=(0,(x_2B_1x_1^t,\dots,x_2B_dx_1^t))=(x_2,y_2)\cdot(x_1,y_1).\]
			Thus, $x_1B_kx_1^t=0$ and $x_1B_kx_2^t=x_2B_kx_1^t$. In particular, $0=e_iB_ke_i^t=B_k[i,i]$ and $B_k[i,j]=e_iB_ke_j^t=e_jB_ke_i^t=B_k[j,i]$ for each vector $e_i,e_j$ of the canonical basis.
		\end{proof}
For the field $\mathbb F_2$, symmetric matrices with zero diagonal are \emph{skew-symmetric}, and we now briefly revisit some fundamental facts about them.
	
\subsection{On skew-symmetric matrices over $\mathbb{F}_2$} 
	A skew-symmetric (or antisymmetric) matrix $B$ over a field $F$ is a square matrix whose transpose equals its negative, i.e.\ it satisfies the condition $B^t=-B$. Clearly, when $F=\mathbb{F}_2$, $B$ is skew-symmetric if and only if $B$ is symmetric and zero-diagonal. We denote by $\Lambda_m$ the matrix subspace of symmetric and zero-diagonal $m\times m$ matrices over $\mathbb{F}_2$.
	
	It is well known that $AXA^t\in \Lambda_m$ for every $m\times m$ matrix $A$ over $\mathbb{F}_2$ and $X\in\Lambda_m$ (see, e.g., Mackey et al.~\cite[Lemma 3.4]{Ma13}). Thus, for each $A\in\gl(m,\mathbb{F}_2)$, the map $X\mapsto AXA^t$ is an automorphism of $\Lambda_m$. In particular, $A\Lambda_mA^t=\Lambda_m$, and for each subspace $\mathcal{B}<\Lambda_m$, we have $\dim_{\mathbb{F}_2}(A\mathcal{B}A^t)=\dim_{\mathbb{F}_2}(\mathcal{B})$.
	
	Any two $m\times m$ matrices $B,C\in\Lambda_m$ of the same rank are congruent, i.e.\ $ABA^t=C$ for some invertible matrix $A$ (\cite[Theorem 5.4]{Ma13}) and, vice versa, if $B$ and $C$ are congruent, then $\rank (B)=\rank (C)$. In particular, two uni-dimensional matrix subspaces of $\Lambda_m$ are congruent if and only if their generators have the same rank.
	
	Recall that a $m\times m$ skew symmetric matrix $B$ has full rank if and only if $m$ is an even integer (\cite[Theorem 5.4]{Ma13}), and in that case, $B$ is called \emph{symplectic}. A symplectic matrix $B$ defines a nondegenerate \emph{symplectic bilinear form} \[\phi_B:\mathbb{F}_2^m\times \mathbb{F}_2^m\longrightarrow\mathbb{F}_2,\quad (x,y)\mapsto x_1Bx_2^t.\] 
	The \emph{symplectic group} $\syp(B)$ of $B$ is the subgroup of $\gl(m,\mathbb{F}_2)$ consisting of all invertible matrices $A$ which preserve $\phi_B$, i.e.\ $\phi_B(x_1A,x_2A)=\phi_B(x_1,x_2)$ for every $x_1,x_2\in\mathbb{F}_2^m$. Thus
	\[\syp(B)=\{A\in\gl(m,\mathbb{F}_2) \mid  ABA^t=B\}.\]
	In particular, any two symplectic matrices $B,C$ are congruent and $\syp(B)$, $\syp(C)$ are isomorphic.

	\subsection{The uni-dimensional case} 
	Let $(R,+,\cdot_\phi) \in$ \baa\ be a binary alternating algebra with underlying vector space $R=\mathbb{F}_2^m\oplus\mathbb{F}_2^d$, uni-dimensional space $R^2=\langle(0,b)\rangle$ and defining matrices $B_1,\dots, B_d\in\langle B\rangle=\{0,B\}$. Since $\phi$ is nondegenerate, then $B$ is a symplectic matrix and $m$ is an even integer.
	
	We denote by $\syp(\phi)$ the symplectic group of the defining matrix $B$ and by \[\fix(b)=\{D\in\gl(d,\mathbb{F}_2) \mid  bD=b\}\] the subgroup of $\gl(d,\mathbb{F}_2)$ of matrices fixing $b$.\\
	
We are ready to illustrate the main result of this section which characterises the automorphism group of a binary alternating algebra $(R,+,\cdot)$ with $\dim(R^2) = 1$ as the semidirect product \[\big(\syp(\phi)\times \fix(b)\big)\ltimes \left\langle\spmat{1_m&C\\0&1_d}:\ C\in\mathbb{F}_2^{m\times d}\right\rangle.\]
The subsequent results are derived as corollaries of Theorem~\ref{uni} and Theorem~\ref{congruent}, respectively.

\begin{corollary}\label{SpFix}
Let $(R,+,\cdot_\phi) \in $ \emph{\baa} be a binary alternating algebra over $R=\mathbb{F}_2^m\oplus\mathbb{F}_2^d$ with uni-dimensional space $R^2=\langle (0,b)\rangle$. Then $\aut(R,+,\cdot_\phi)$ is the semidirect product
		\[\left(\syp(\phi)\times \fix(b)\right)\ltimes \left\langle\spmat{1_m&C\\0&1_d}:\ C\in\mathbb{F}_2^{m\times d}\right\rangle,\]
			i.e.\
		\[\aut(R,+,\cdot_\phi)= \left\{\left[\begin{array}{c|c}A&C\\\hline 0&D\end{array}\right]:\, A\in \syp(\phi),\, C\in \mathbb{F}_2^{m\times d},\,D\in \fix(b)\right\}.\]
	\end{corollary}
		\begin{proof}
The claim is obtained by applying Theorem~\ref{uni} to the case of binary alternating algebras.	
		\end{proof}
\begin{corollary}\label{singleclass}
There is a single isomorphism class of binary alternating algebras with underlying vector space $\mathbb{F}_2^m\oplus\mathbb{F}_2^d$ and uni-dimensional space $R^2$.
\end{corollary}
\begin{proof}
The claim follows from Theorem~\ref{congruent} and by observing that any two $m\times m$ symplectic matrices over $\mathbb{F}_2$ are congruent. 
\end{proof}

		\begin{remark}\label{isouni}
			Referring to Corollary~\ref{singleclass}, we explicitly observe  that the matrices of the  form
			\[\left[\begin{array}{c|c}
				\syp(\phi_1,\phi_2)&\mathbb{F}_2^{m\times d}\\
				\hline
				0_{d\times m}&\{D:\, b_1D=b_2\}
				\end{array}\right]\]
			represent the set of isomorphisms of two binary alternating algebras $(R_1,+,\cdot_1)$, $(R_2,+,\cdot_2)$ with underlying vector space $R_1=R_2=\mathbb{F}_2^m\oplus\mathbb{F}_2^d$ and uni-dimensional spaces $R_1^2=\langle(0,b_1)\rangle$ and $R_2^2=\langle(0,b_2)\rangle$,  where $\syp(\phi_1,\phi_2)$ denotes the set of $m\times m$ invertible matrices $A$ such that $AB_2A^t=B_1$ and where $\phi_i$ and $B_i$ represent the alternating bilinear maps and the defining matrices of $R_i$.
			
		\end{remark}

\begin{example}
	Let us consider the matrix of the \emph{standard alternating form} of rank 6 over the field with two elements:
	\[B=\pmat{0&0&0&0&0&1\\0&0&0&0&1&0\\0&0&0&1&0&0\\0&0&1&0&0&0\\0&1&0&0&0&0\\1&0&0&0&0&0}.\]
The sequences $(B,B)$, $(B,0)$, $(0,B)$ define isomorphic binary alternating algebras over $R=\mathbb{F}_2^8\simeq\mathbb{F}_2^6\oplus\mathbb{F}_2^2$ with $\Ann(R)=\langle e_7,e_8\rangle$ and uni-dimensional space $R^2$ generated by $\langle e_7+e_8\rangle$, $\langle e_7\rangle$, $\langle e_8\rangle$ respectively. The automorphism group of these algebras are determined by $\syp(B)$ and by the groups of order two generated by $\pmat{0&1\\1&0}$, $\pmat{1&0\\1&1}$ and $\pmat{1&1\\0&1}$ respectively.
	
	The set of isomorphisms between the algebra defined by $(B,B)$ and the algebra defined by $(0,B)$ is determined by $\syp(B)$ and by the set \[\left\{\pmat{1&0\\1&1},\pmat{1&1\\0&1}\right\}.\]
\end{example}

\subsection{The \baa\ version of Theorem \ref{theta}}
In this section, we obtain an easy proof of Theorem~\ref{theta} by means of \baa s.
	Let us start by emphasising the connection between the matrices $\lambda_{e_i}$  of a binary bi-brace $(R,+,\circ)$, as described in Eq.~\eqref{lambda}, and the defining matrices $B_1,\dots, B_d$ of the corresponding binary alternating algebra $(R,+,\cdot_\phi)$, both operating over the same $\mathbb{F}_2$-vector space $R=\mathbb{F}_2^m\oplus \mathbb{F}_2^d$,
	where
	\begin{align*}
		(x_1,y_1)\lambda_{(x_2,y_2)}+(x_2,y_2)=(x_1,y_1)\circ(x_2,y_2)=(x_1,y_1)+(x_2,y_2)+(0,\phi(x_1,x_2)).
	\end{align*}
	We have
	\begin{align*}
		(x_1,y_1)\lambda_{(x_2,y_2)}&=(x_1,y_1)+(0,\phi(x_1,x_2))\\
		&=(x_1,y_1)+(0,(x_1B_1x_2^t,\dots,x_1B_dx_2^t))\\
		&=(x_1,y_1)1_{m+d}+(0,x_1\pmat{B_1x_2^t&\dots&B_dx_2^t})\\
		&=(x_1,y_1)\left(\pmat{1_{m}&0_{m\times d}\\0_{d\times m}&1_{d}}+\pmat{0_{m}&\pmat{B_1x_2^t&\dots&B_dx_2^t}\\0_{d\times m}&0_{d\times d}}\right).
	\end{align*}
	Therefore,
	\begin{align*}
		\lambda_{(x_2,y_2)}=\pmat{1_{m}&\pmat{B_1x_2^t&\dots&B_dx_2^t}\\0_{d\times m}&1_{d}},\quad (x_2,y_2)\in\mathbb{F}_2^m\oplus\mathbb{F}_2^d.
	\end{align*}
	So,  writing with respect to the canonical basis of $\mathbb{F}_2^m$, we can represent  
	\[\lambda_{e_i}=\pmat{1_m&\pmat{B_1e_i^t&\dots&B_de_i^t}\\0_{d\times m}&1_d}\]
	for each $1 \leq i \leq m$. The same result has been obtained in the context of \ear~\cite[Theorem 3.11]{calderini2021properties}).
The matrix \[\Theta=\pmat{\Theta_1&\dots&\Theta_m},\]
whose columns are defined for each $1\leq i\leq m$ by  \[\Theta_i=\pmat{B_1e_i^t&\dots&B_de_i^t},\] is precisely the defining matrix of an operation $\circ$ as introduced by Civino et al.~\cite[Definition 3.5]{CBS19} in the context of alternative differential cryptanalysis.
We can then revisit the following result:
	\begin{theorem}
		Let $\Theta_1,\dots,\Theta_m$ be $m\times d$ matrices over $\mathbb{F}_2$ and $\Theta_{i,j}$ the $d$-dimensional row vector of the matrix $\Theta_i$, $i,j\in \{1,\dots,m\}$. Then $\Theta=\pmat{\Theta_i,\dots,\Theta_m}$ is the defining matrix of a binary bi-brace $(\mathbb{F}_2^m\oplus\mathbb{F}_2^d,+,\circ)$ if and only if, for each $i,j$,
		the following conditions hold:
		\begin{enumerate}[(i)]
			\item \label{civ1} $\Theta_{i,i}=0$;
			\item \label{civ2} $\Theta_{i,j}=\Theta_{j,i}$;
			\item \label{civ3}$\Theta_1,\dots,\Theta_d$ are linearly independent.
		\end{enumerate}
	\end{theorem}
\begin{proof}
			Let us consider the $m\times m$ matrices $B_1,\dots,B_d$ defined for each $1 \leq i \leq m$ by
			\[\pmat{B_1e_i^t&\dots&B_de_i^t}=\Theta_i.\] We have that \[\Theta_{i,j}=e_j\pmat{B_1e_i^t&\dots&B_de_i^t}=(e_jB_ie_i^t,\dots,e_jB_de_i^t).\]
The conditions described by equations ~\eqref{civ1} and \eqref{civ2} are satisfied if and only if $B_1,\dots, B_d$ are symmetric matrices with zero diagonal elements.  
Meanwhile, condition \eqref{civ3} is fulfilled only when the concatenated matrix $\pmat{B_1&\dots&B_d}$ possesses a rank of $m$.
 This means that \eqref{civ1}, \eqref{civ2}, \eqref{civ3} are equivalent to requiring that the bilinear map $\phi:\mathbb{F}_2^m\oplus\mathbb{F}_2^m\longrightarrow\mathbb{F}_2^d$ determined by $B_1,\dots, B_d$ is alternating and nondegenerate. This is equivalent to defining a nilpotent algebra of class two and nil index two, and therefore to defining a binary bi-brace.
	\end{proof}

\begin{example}
	The following $4\times 16$ matrix over $\mathbb{F}_2$ defined as
	\[\Theta=\pmat{\Theta_1&\Theta_2&\Theta_3&\Theta_4}\\=\left[\begin{array}{c|c|c|c}
		\mat{0& 0& 0& 0\\ 		 			
			0& 0& 0& 1\\ 					
			1& 0& 0& 1\\ 					
			1& 0& 1& 1}&   		  	 		
		\mat{0& 0& 0& 1\\
			0& 0& 0& 0\\ 
			1& 1& 1& 1\\ 
			0& 1& 1& 1}& 
		\mat{1& 0& 0& 1\\ 
			1& 1& 1& 1\\ 
			0& 0& 0& 0\\ 
			1& 0& 0& 1}&  
		\mat{1& 0& 1& 1\\
			0& 1& 1& 1\\
			1& 0& 0& 1\\
			0& 0& 0& 0}\end{array}\right]\]
is the defining matrix of the operation 
	\[(x_1,y_1)\circ(x_2,y_2)=(x_1+x_2,y_1+y_2+\phi(x_1,x_2)),\quad (x_i,y_i)\in\mathbb{F}_2^4\oplus\mathbb{F}_2^4\]
	where $\phi$ is the bilinear map of Example~\ref{example1} with defining matrices $(B_1,B_2,B_3,B_4)$. Note that $\Theta_i=\pmat{B_1e_i^t&B_2e_i^t&B_3e_i^t&B_4e_i^t}$ for each $1\leq i\leq 4$.
\end{example}

\subsection{Direct sum of \baa}
The practicality of the alternative differential attack can be enhanced by assuming round independence, much like in the classical scenario. This allows for the "factorisation" of $\circ$-differentials propagation through encryption functions across all layers. If we aim to analyse the difference propagation to each s-box individually, then attacking the cipher through a parallel operation is necessary. This operation should behave identically across all the $h$ s-boxes, necessitating that the diffusion layer function $\mu$ is an automorphism of such constructed parallel operations. The objective of this final section is to elucidate the computation of these maps. We show how to extend the findings of Corollary~\ref{cor:aut} and Theorem~\ref{uni} to the direct sum of bi-braces.

Let $h\geq 2$ be an integer and let $(R,+,\cdot_\phi)$ be a binary alternating algebra, $R=V\oplus W$, $W=\Ann(R)\simeq \mathbb{F}_2^d$, $V\simeq\mathbb{F}_2^m$. Let us consider the binary alternating algebra $\left(S,+,\cdot\right)$ of dimension $(m+d)^h$ over $\mathbb{F}_2$, given by the direct $h$th sum $S=\bigoplus_{i=1}^h R$ where:
\begin{align*}
	(x_1,\dots,x_h)\cdot(y_1,\dots,y_h)=(x_1\cdot_\phi y_1,\dots,x_h\cdot_\phi y_h)
\end{align*}
for every $x_i,y_i\in R$.
Notice that $\Ann(S)=\bigoplus^h W$ and $S^2=\bigoplus^h R^2$.

Let $g_{i,j}$ be square matrices of size $m+d$ for each $1 \leq i,j \leq h$ and $G=\pmat{g_{i,j}}\in\gl(S)$. Clearly $G$ is an automorphism of $S$ if and only if $G$ is invertible and
\begin{align}\label{eq:long}
	(x_1,\dots,x_h)\pmat{g_{i,j}}\cdot(y_1,\dots,y_h)\pmat{g_{i,j}}=(x_1\cdot_\phi y_1,\dots,x_h\cdot_\phi y_h)\pmat{g_{i,j}}
\end{align}
for every $x_i,y_i\in R$. Therefore,
\begin{align*}
	\mathrm{LHS} \text { of Eq.~\eqref{eq:long}}=&\left(\sum_{i}{x_ig_{i,1}},\dots,\sum_i{x_ig_{i,h}}\right)\cdot\left(\sum_{l}{y_lg_{l,1}},\dots,\sum_l{y_lg_{l,h}}\right)\\
	=&\left(\sum_{i,l}{x_ig_{i,1}\cdot_\phi y_lg_{l,1}},\dots,\sum_{i,l}{x_ig_{i,h}\cdot_\phi y_lg_{l,h}}\right)=\\
	\mathrm{RHS} \text { of Eq.~\eqref{eq:long}}=&\left(\sum_i{(x_i\cdot_\phi y_i)g_{i,1}},\dots,\sum_i{(x_i\cdot_\phi y_i)g_{i,h}}\right),
\end{align*}
meaning that
\begin{align*}
	\sum_{i,l}{x_ig_{i,j}\cdot_\phi y_lg_{l,j}}=\sum_i{(x_i\cdot_\phi y_i)g_{i,j}}
\end{align*}
for each $1 \leq i,j,l\leq h$ and $x_i,y_i,y_l\in R$. We can rewrite the last equation as
\begin{align}\label{eq5}
	\sum_{i}{x_ig_{i,j}\cdot_\phi y_ig_{i,j}}+\sum_{i\neq l}{x_ig_{i,j}\cdot_\phi y_lg_{l,j}}=\sum_i{(x_i\cdot_\phi y_i)g_{i,j}},
\end{align}
obtaining that Eq.~\eqref{eq5} holds when $x_2=\dots=x_h=0=y_2=\dots=y_h$, i.e.\
\[x_1g_{1,j}\cdot_\phi y_1g_{1,j}=(x_1\cdot_\phi y_1)g_{1,j}.\]
Thus, $g_{1,j}$ is an endomorphism of the algebra $R$. By iterating this argument, we can prove that $g_{i,j}$ are endomorphisms of $R$ for every $i,j$. Therefore, Eq.~\eqref{eq5} becomes
\begin{align*}
	\forall i\neq l\quad {x_ig_{i,j}\cdot_\phi y_lg_{l,j}}=0,
\end{align*}
which is equivalent to the  condition
\begin{align}\label{condition1}
	\forall z\in R,\ 1 \leq i,j,l \leq h, i\neq l:\quad zg_{i,j}\in\Ann(R) \text{ or } zg_{l,j}\in\Ann(R).
\end{align}
Let us write $g_{i,j}$ in the block form as in Corollary~\ref{cor:aut}. Setting $z=(x_1,y_1)\in V\oplus W$ we have
\[zg_{i,j}=(x_1,y_1)\pmat{A_{i,j}&C_{i,j}\\0&D_{i,j}}=(x_1A_{i,j},x_1C_{i,j}+y_1D_{i,j}).\]
Note that, for two fixed indices $1 \leq i,j \leq h$, if $zg_{i,j}\in\Ann(R)=0_m\oplus\mathbb{F}_2^d$ for each $z\in R$, then 
$A_{i,j}=0$. So, Eq.~\eqref{condition1} is equivalent to:
\begin{align}\label{condition2}
 \forall	1 \leq i,j,l\leq h, i\neq l:\quad A_{i,j}=0\text{ or } A_{l,j}=0.
\end{align}
From the invertibility of 
\[G=\pmat{\mat{A_{11}&C_{11}\\0&D_{11}}&\dots&\mat{A_{1k}&C_{1h}\\0&D_{1h}}\\
	\vdots&&\vdots\\
	\mat{A_{h1}&C_{h1}\\0&D_{h1}}&\dots&\mat{A_{hh}&C_{1
			hh}\\0&D_{hh}}}\]
and by Eq.~\eqref{condition2}, it follows that for each column index $1 \leq j \leq h$ there exist exactly one row index $j\pi$, for some $\pi\in\sym(h)$, such that $A_{j\pi,j}\neq 0$ and $A_{j\pi,l}=0$ for each $j\neq l$.
Since $g_{j\pi,l}=\pmat{A_{j\pi,l}&C_{j\pi,l}\\0&D_{j\pi,l}}$ is an endomorphism of $(R,+,\cdot_\phi)$, if $j\neq l$ we have
\[\phi(x_1,x_2)D_{j\pi,l}=\phi(x_1A_{j\pi,l},x_2A_{j\pi,l})=0\]
and $A_{j\pi,j}\in\gl(m,\mathbb{F}_2)$.
Note also that, if $R^2=\langle(0,b)\rangle$ is uni-dimensional, then $bD_{j\pi,l}=0$ for $j\neq l$.

Now, let us consider a matrix $\beta\in\gl((m+d)^h,\mathbb{F}_2)$ such that
\[\beta G\beta^{-1}=\left[\begin{array}{c|c}
	\mat{A_{11}&\dots&A_{1h}\\
		\vdots&&\vdots\\
		A_{h1}&\dots&A_{hh}}&\mat{C_{11}&\dots&C_{1h}\\
		\vdots&&\vdots\\
		C_{h1}&\dots&C_{hh}}\\
	\hline
	0&\mat{D_{11}&\dots&D_{1h}\\
		\vdots&&\vdots\\
		D_{h1}&\dots&D_{hh}}
\end{array}\right]\]
Then $\beta G\beta^{-1}$ is an automorphism of the binary alternating algebra $(S,+,\cdot_\beta)$, where 
\[x\cdot_\beta y=(x\beta\cdot y\beta)\beta^{-1},\]
for every $x,y\in S$. Clearly $(S,+,\cdot_\beta)$ is a binary alternating algebra isomorphic to $(S,+,\cdot)$. Since $\beta G\beta^{-1}$ is invertible, the matrices $A=\pmat{A_{i,j}}$ and $D=\pmat{D_{i,j}}$ are invertible.

From the previous argument we obtain the following characterisation of the automorphism group of the direct sum of algebras in \baa, generalising 
the one obtained by Calderini et al.~\cite{calderini2024optimal} in the less general context of bi-braces with maximal socle.
\begin{theorem}\label{thm:parallelgen}
	Let $(S,+,\cdot)$ be the direct $h$th sum of a binary alternating algebra $(R,+,\cdot_\phi)$, with bilinear map $\phi:V\times V\longrightarrow W$, $R=V\oplus W$. The automorphism group of $(S,+,\cdot)$ consists of the matrices
	\[G=\pmat{g_{ij}}=\pmat{\mat{A_{11}&C_{11}\\0&D_{11}}&\dots&\mat{A_{1k}&C_{1h}\\0&D_{1h}}\\
		\vdots&&\vdots\\
		\mat{A_{h1}&C_{h1}\\0&D_{h1}}&\dots&\mat{A_{hh}&C_{
				hh}\\0&D_{hh}}}\]
	such that for a given permutation $\pi=\pi_G\in\sym(h)$ and for each $1 \leq i,j\leq h$ the following conditions hold:
	\begin{enumerate}
		\item $g_{j\pi, j}\in \aut(R,+,\cdot_\phi)$;
		\item if $j\neq i$ then $A_{j\pi,i}=0$ and $\phi(x_1,x_2)D_{j\pi,i}=0$ for each $x_1,x_2\in V$;
		\item $D=\pmat{D_{i,j}}$ is an invertible matrix.
	\end{enumerate}
	In particular, if $R^2=\langle(0,b)\rangle$ is a uni-dimensional space, then $A_{j\pi,j}\in \syp(\phi)$, and $bD_{j\pi,j}=b$ for each $1 \leq j\leq h$.
\end{theorem}

The following final result describes the automorphism group of the direct sum of $h$ distinct binary alternating algebras defined by the same parameters $m$ and $d$, all of them with uni-dimensional $R^2$ space.

\begin{theorem}\label{thm:parallelone}
	Let $\{(R_i,+,\cdot_i)\}_{i=1}^h$ be $h\geq 2$ binary alternating algebras, $R_i=\mathbb{F}_2^m\oplus\mathbb{F}_2^d$, with uni-dimensional spaces $R_i^2=\langle(0,b_i)\rangle$, alternating bilinear maps $\phi_i$ and defining matrices $B_i$. The automorphism group of the direct sum of algebras $(S,+,\cdot)$, $S=R_1\oplus \dots\oplus R_h$, where
	\begin{align*}
		(x_1,\dots,x_h)\cdot(y_1,\dots,y_h)=(x_1\cdot_1y_1,\dots,x_h\cdot_h y_h),
	\end{align*}
	consist of the matrices $G=\pmat{g_{i,j}}$, $g_{i,j}=\pmat{A_{i,j}&C_{i,j}\\0&D_{i,j}}$ such that for a given permutation $\pi=\pi_G\in \sym(h)$ and for each $1 \leq i,j \leq h$ the following conditions hold:
	\begin{enumerate}
		\item \label{eq:cond1} $A_{j\pi,j}B_jA_{j\pi,j}^t=B_i$ and $b_iD_{j\pi,j}=b_j$;
		\item  if $j\neq i$ then $A_{i\pi,j}=0$ and $bD_{i\pi,j}=0$;
		\item  $D=\pmat{D_{i,j}}$ is an invertible matrix.
	\end{enumerate}
\end{theorem}
	\begin{proof}
	A matrix $G=\pmat{g_{i,j}}\in\gl(S)$ is an automorphism of $(S,+,\cdot)$ if and only if $G$ is invertible and the following equations hold for each $x,y\in \mathbb{F}_2^m\oplus\mathbb{F}_2^d$ and $1 \leq i,j,l \leq h$:
	\begin{itemize}
		\item $(x\cdot_i y)g_{i,j}=xg_{i,j}\cdot_j yg_{i,j}$;
		\item $\forall\ i\neq l\quad xg_{i,j}\cdot_jyg_{l,j}=0$.
	\end{itemize}
	The first equation implies that $g_{i,j}: R_i\longrightarrow R_j$ is an homomorphism of algebras for each pair of indices $i,j$. Moreover, arguing as in the previous theorem, for each column index $j$ there exist exactly one row index $j\pi$, $\pi\in\sym(h)$, such that $g_{j\pi,j}$ is invertible. By Corollary~\ref{singleclass} and Remark~\ref{isouni}, this is equivalent to the condition~\eqref{eq:cond1} in the statement.
\end{proof}
\section{Conclusion}\label{sec:concl}
In this paper we have introduced the concepts of binary bi-braces and binary alternating algebras. Working on a finite-dimensional binary vector space $(M,+)$ serving as the message space within an SPN, which is partitioned into bricks of size corresponding to the s-boxes and denoted as $M = R \oplus R \oplus \dots \oplus R$, these structures corresponds to pairs of translation groups $(T_+,T_\circ)$ as in Calderini et al.~\cite{calderini2021properties}, where the operation $\circ$ is employed to compute differentials for an alternative differential attack. Notably, we have observed the equivalence among the three algebraic constructions (see\ Remark~\ref{rmk:equiv}). Through the examination of binary alternating algebras using bilinear maps, we have derived in Sect.~\ref{sec:algebras} and Sect.~\ref{sec:binalg} a useful description of the automorphism group of the algebra $(R,+,\cdot)$, both in the general scenario and particularly when $\dim(R^2) = 1$, where the predictability of differences subsequent to key addition have the highest probabilities. We have also shown how the knowledge of $\aut(R,+,\cdot)$ can used to obtain the description of all the automorphisms of $(M,+,\cdot)$ (see\ Theorem~\ref{thm:parallelgen} and Theorem~\ref{thm:parallelone}).

The abstract algebraic entity $\aut(M,+,\cdot)$ contains a subset of matrices embodying a trapdoor diffusion layer, exploitable by attackers via the operation $\circ$ induced by the algebra's product. 

It is noteworthy that in preceding sections, we have meticulously explored the analysis of $\circ$-difference propagations across the diffusion layer and the key-addition layer of the cipher. However, we have not addressed the examination of difference propagation through the confusion layer.
This omission is justified by the inherent XOR-nonlinearity of the s-box in an XOR-based SPN, implying a likely deterioration in its differential properties when examined with respect to a different operation. 
This assertion finds support, for example, examining the differential properties of the  optimal 4-bit s-boxes as classified by Leander and Poschmann~\cite{leander2007classification}, where each equivalence class contains maps exhibiting varying degrees of nonlinearity concerning operations derived from binary bi-braces. Particularly, among the 16 classes, four exhibit the lowest level of nonlinearity~\cite{calderini2024optimal}.

Using the results of this research, creating a trapdoor cipher within the SPN family, susceptible to the alternative differential attack, can be achieved through the following steps:
\begin{itemize}
\item select an s-box $\gamma'$ on $R$ showing good differential properties w.r.t.\ the XOR;
\item take an operation $\circ$ (i.e.\ a binary bi brace, i.e.\ an alternative binary algebra) such that the considered s-box exhibits a biased distribution of differences computed w.r.t.\ $\circ$. A reasonable assumption, though not entirely restrictive, would be to set $\dim(R^2)=1$; 
\item determine $\aut(R,+,\circ)$ following the method outlined in Corollary~\ref{SpFix};
\item compute $\aut(M,+,\circ)$, where $\circ$ denotes the parallel extension of the operation to $M$, based on the knowledge of $\aut(R,+,\circ)$ using Theorem~\ref{thm:parallelone};
\item search in $\aut(M,+,\circ)$ for a map $\mu$ exhibiting good diffusion properties.
\end{itemize}
Computational experiments indicate that the devised cipher, coupled with the selected operation, could constitute a trapdoor cipher along with its corresponding trapdoor. This assertion is supported by the finding that virtually any selection of $\mu \in \aut(M,+,\circ)$ could result in a cipher vulnerable to the alternative differential attack~\cite{calderini2024optimal}.

Potential countermeasures against alternative differential cryptanalysis can be inferred directly from Theorem~\ref{thm:parallelgen}, which highlights that a Maximum Distance Separable (MDS) diffusion layer prevents differences with respect to a parallel alternative operation defined on $M$ from propagating. Additionally, empirical investigations indicate that the impact of alternative operations on the differential behaviour of 8-bits s-boxes is less severe compared to their effect on 4-bit s-boxes. Nonetheless, the realm of possibly vulnerable ciphers employing 4-bit s-boxes and nonMDS diffusion layers is widespread, especially due to the presence of lightweight ciphers.


\bibliographystyle{amsalpha}
\bibliography{biblio}

\end{document}